\documentclass[a4paper,11pt]{amsart}

\usepackage[titletoc,title]{appendix}
\usepackage{amsmath,amssymb,amsthm,latexsym,amscd,mathrsfs}
\usepackage{indentfirst}
\usepackage{stmaryrd}
\usepackage{graphicx}
\usepackage{extarrows}
\usepackage{multirow}
\usepackage{latexsym}
\usepackage{amsfonts}
\usepackage{color}
\usepackage{pictexwd,dcpic}
\usepackage{graphicx}
\usepackage{psfrag}
\usepackage{hyperref}
\usepackage{comment}
\usepackage{enumerate}
\usepackage{romannum}
\usepackage{tikz}
\usepackage{pgfplots}
\pgfplotsset{compat=1.18}
\usepackage{hyperref,cleveref}

\numberwithin{equation}{section} \theoremstyle{plain}
\newtheorem{thm}{Theorem}[section]

\newtheorem{lem}[thm]{Lemma}
\newtheorem{cor}[thm]{Corollary}
\newtheorem{defn}[thm]{Definition}

\makeatletter

\def\<{\langle}
\def\>{\rangle}
\def\({\left(}
\def\){\right)}
\def\[{\left[}
\def\]{\right]}

\makeatother
\title[ Minimal Hypersurfaces with constant scalar curvature in $\mathbf{S}^6$ ]
{Minimal Hypersurfaces with constant scalar curvature in $\mathbf{S}^6$}
\author[Y. Tao]{Ya Tao}
\address{Chern Institute of Mathematics and LPMC,
Nankai University, Tianjin 300071, P. R. China.}
\email{tao-ya@mail.nankai.edu.cn}

\date{}
\thanks{Y. Tao is partially supported by 
	NSFC (No. 12371048) and Nankai Zhide
	Foundation.}
\begin{document}
\maketitle

%%%%%%%%%%%%%%%%%%%%%
\begin{abstract}
%In this paper, we prove that for any closed minimal $5$-dimensional hypersurface $M^5$
%in the unit sphere $\mathbf{S}^6$ for which the functions $S,f_3,f_4$ are constants, and at points with at least four distinct principal curvatures, or exactly three with multiplicities $(2,2,1)$, the principal curvatures satisfy $A(r)>0$ for $r=1,2,\cdots,5$, the hypersurface $M^5$ is isoparametric.
In this paper, we propose certain assumptions on the principal curvatures for a closed minimal hypersurface $M^5$ in  $\mathbf{S}^6$ to be isoparametric, provided that the functions $S, f_3,f_4$ are constants.
Our result removes the nonnegative scalar curvature assumption as in Tang and Yan \cite{TY}. 
Finally, as a rigidity result, if $M^5\subset \mathbf{S}^6$ has a point with exactly two distinct principal curvatures, then it must be a Clifford torus.
\end{abstract}
%%%%%%%%%%%%%%%%%%%%%%%

%%%%%%%%%%%%%%%%%%%%%%
	\section{Introduction}
	
Exploring the relationship between geometric invariants and the structure of manifolds or submanifolds has always been an important problem in global differential geometry. In 1968, Simons \cite{Sim} gave an integral formula for the squared norm $S$ of the second fundamental form. Shortly thereafter, Chern et al. \cite{CdK} and Lawson \cite{Law} independently obtained the rigidity result when $S = n$, as shown in the theorem below.
\begin{thm}$($\cite{CdK,Law,Sim}$)$
	Let $M^n\subset\mathbf{S}^{n+1}$ be a closed minimal immersed hypersurface. Then 
	$$\int_{M}(S-n)S\geq 0.$$
	In particular, for $S\leq n$, one has either $S\equiv0$ or $S\equiv n$ on $M^n$. Moreover, for $S\equiv0$, $M^n$ is the equatorial $n$-spheres in $\mathbf{S}^{n+1}$; for $S\equiv n$, $M^n$ is the Clifford tori $\mathbb{S}^k\left(\sqrt{\frac{k}{n}}\right)\times\mathbb{S}^{n-k}\left(\sqrt{\frac{n-k}{n}}\right)$, $1\leq k\leq n-1$.
\end{thm}

Based on the above work, Chern \cite{CSS} proposed the following famous conjecture regarding compact minimal hypersurfaces in a sphere.\\
\noindent \textbf{Chern conjecture.} \quad Let $M^n$ be an $ n$-dimensional compact minimal hypersurface in the unit sphere $\mathbf{S}^{n+1}$ with constant scalar curvature $R$. Then the set of all possible values of the scalar curvature (equivalently, $ S $) of $ M^n $ is a discrete set in $ \mathbb{R} $.

Chern's conjecture can be decomposed into several problems of pinching the scalar curvature of compact minimal hypersurfaces in the unit sphere. The first difficulty is the second gap problem, which states that, under the assumptions of Chern conjecture, if $n \leq  S\leq 2n $, then either $S = n$ or $ S = 2n$.  Peng and Terng \cite{PT1, PT2} completely solved the second gap problem for the case $n = 3$, and $S = 6$ can be realized by Cartan minimal isoparametric hypersurfaces in the unit sphere $\mathbf{S}^4 $. Furthermore, for general dimension $n$, under the assumptions of Chern conjecture, they also proved that if $S > n $, then necessarily $S > n + \frac{1}{12n}$, which for the first time gave a breakthrough in the second pinching problem. Subsequently, Yang and Cheng \cite{YC} advanced the second gap problem to $\frac{n}{3}$; Suh and Yang \cite{SY} improved Peng and Terng's result to $ \frac{3n}{7}$.  For a more detailed introduction to Chern conjecture and related problems, we refer the reader to \cite{GT, LXX,XX, TWY,TY}.

Up to now, all known closed minimal hypersurfaces in spheres with constant scalar curvature are isoparametric. Based on this, Verstraelen, Montiel, Ros and Urbano \cite{VL} first proposed a stronger version of Chern conjecture, namely:\\
\noindent \textbf{Stronger Chern conjecture.} \quad Let $M^n$ be a closed, minimally immersed hypersurface of the unit sphere $\mathbf{S}^{n+1}$ with constant scalar curvature. Then $M^n$ is isoparametric.

In 1993, Chang \cite{CSP} proved the above version of Chern's conjecture for the case $n = 3$. In fact, without requiring minimality, de Almeida and Brito \cite{dB} proved the following theorem. 
\begin{thm}$($\cite{dB}$)$
	Let $M^3\subset\mathbf{S}^4$ be a closed hypersurface with constant mean curvature $H$ and constant nonnegative scalar curvature $R$. Then $M^3$ is isoparametric.
\end{thm}
Soon thereafter, Chang \cite{CSP93}, Cheng and Wan \cite{CW} independently proved that under the assumptions of the above theorem, one always has $R \geq 0$, thereby generalizing the aforementioned theorem.

In the case $n=4$, Lusala et al. \cite{LSS,MS} proved that closed minimal Willmore hypersurfaces with nonnegative constant scalar curvature in $\mathbf{S}^5$ are isoparametric, where the Willmore assumption here is equivalent to $f_3 = 0$. Deng et al. \cite{DGW} removed the nonnegative scalar curvature assumption and generalized this result. In addition, Tang and Yang \cite{TYang} proved that if the number of distinct principal curvatures is fixed, then any closed minimal hypersurface $M^n \subset \mathbf{S}^{n+1}$  with constant $3$-rd mean curvature $H_3$ and constant nonnegative scalar curvature $R$ must be isoparametric. 

In the case $n = 6$, Scherfner et al. \cite{SVW} proved that closed hypersurfaces in $\mathbf{S}^7$ with constant nonnegative scalar curvature are isoparametric if $H = f_3 = f_5 = 0$ and $f_4 = \text{const.}$, which is listed in \cite[Theorem 6]{SWY}.

For general dimension $n$, based on the method of \cite{dB}, Tang, Wei and Yan \cite{TWY} and Tang and Yan \cite{TY} proved the following theorem.
\begin{thm}$($\cite{TY}$)$\label{thTY}
	Let $M^n$ $(n>3)$ be a closed hypersurface in the unit sphere $\mathbf{S}^{n+1}$. If $R\geq 0$ and $\sum_{i=1}^{n}\lambda_{i}^k$ $(k=1,\cdots,n-1)$ are constants for principal curvatures $\lambda_{1}\leq\lambda_{2}\leq \cdots \leq \lambda_{n}$, then $M^n$ is isoparametric.
\end{thm}
When $ n = 4$, the result of the above theorem was also obtained in \cite{SX}. Furthermore, for the minimal case, Cheng and Li \cite{CL} proved that if the number of distinct principal curvatures is constant, the assumption $ R \geq 0 $ on $M^4$  in Theorem \ref{thTY} is redundant. 
Based on the result of Cheng and Li \cite{CL},  using the method in \cite{dB}, He, Xu and Zhao \cite{HXZ} removed the requirement on the number of distinct principal curvatures in \cite{CL}. This also shows, in the four dimensional minimal case, that Theorem \ref{thTY} of Tang and Yan does not require the assumption of nonnegative scalar curvature. Therefore, one can naturally propose the following question: for general dimensions, whether the assumption $ R \geq 0 $ in Theorem \ref{thTY} can also be removed. This would bring the result closer to the statement of the Chern conjecture.

%To the best of the author's knowledge, Theorem \ref{thTY} provides the best known result for this type of rigidity problem for $5$-dimensional hypersurfaces in the unit sphere $\mathbf{S}^6$.

In this paper, for $ n = 5$ in the minimal case, we propose a new assumption on the principal curvatures (see $(\ref{Ass})$) that can replace the requirement of nonnegative scalar curvature  in Theorem \ref{thTY}. The method that we use generalizes the 3-form $\Phi$ in \cite{CL} to an $(n-1)$-form (see $(\ref{Phi})$)  and performs a crucial simplification of the differential of this $(n-1)$-form when $n=5$ (see Section \ref{simp}). %Moreover, Assumption $(\ref{Ass})$ only needs to hold at points with at least four distinct principal curvatures, or exactly three with multiplicities $(2,2,1) $; see Theorem \ref{main thm}. Besides, we show that under several distributions of principal curvatures, Assumption $(\ref{Ass})$ holds. %Therefore, for the hypersurface $ M^5 \subset \mathbf{S}^6$, provided that at every point with at least four distinct principal curvatures, or exactly three with multiplicities $(2,2,1)$, the principal curvatures fall into these specific patterns, we remove the nonnegative scalar curvature assumption in Theorem \ref{thTY}. 
Throughout this paper, we adopt the conventions that $\sigma_3\geq 0$ and $\lambda_{1}\leq \lambda_{2}\leq \lambda_{3}\leq \lambda_{4}\leq \lambda_{5}.$ Otherwise, we consider $\tilde{\lambda_{i}}=-\lambda_{i}$ by choosing the normal vector field in the opposite direction. For three distinct indices $1\leq i<j<k\leq 5$,  we define
%	\par\noindent
%	\vbox{\hsize=\linewidth\centering
	%$\displaystyle\begin{aligned}
		\begin{equation}\label{s1}
			s_1^{ijk}=\lambda_{i}+\lambda_j+\lambda_k,\;\;s_2^{ijk}=\lambda_i\lambda_j+\lambda_i\lambda_k+\lambda_j\lambda_k,\;\;s_3^{ijk}=\lambda_i\lambda_j\lambda_k,
		\end{equation}
		\begin{gather}\label{s}
			s^{ijk} =(\lambda_{i}-\lambda_{j})^2(\lambda_{i}-\lambda_{k})^2(\lambda_{j}-\lambda_{k})^2.
		\end{gather}
		%	\end{aligned}$\par}
	For $r=1,2,\cdots,5$, we set $I_r=\{1,2,3,4,5\}\setminus\{r\}$. 
	
	%Let $s^{123}_1=\lambda_{1}+\lambda_{2}+\lambda_{3}$, $s^{123}_2=\lambda_{1}\lambda_{2}+\lambda_{2}\lambda_{3}+\lambda_{1}\lambda_{3}$ and $s^{123}_3=\lambda_{1}\lambda_{2}\lambda_{3}$ denote the elementary symmetric polynomials in the eigenvalue $\lambda_{1},\lambda_{2},\lambda_{3}$ respectively.

The main result of this paper is the following theorem.
\begin{thm}\label{main thm}
	Let $M^5$ be a closed $5$-dimensional minimal hypersurface in the unit sphere $\mathbf{S}^6$ such that $S, f_3,f_4$ are constants. Suppose in addition 
	%\begin{itemize}
	%\item[$(1)$] $S, f_3,f_4$ are constants;
	%\item[$(2)$] $\operatorname{tr}(\mathcal{A})=0$, $\operatorname{tr}(\mathcal{A}^k)$ are constants for $k=2,3,4$;
	%\item[$(2)$]%Let $s_1^{ijk}=\lambda_{i}+\lambda_j+\lambda_k,s_2^{ijk}=\lambda_i\lambda_j+\lambda_i\lambda_k+\lambda_j\lambda_k,s_3^{ijk}=\lambda_i\lambda_j\lambda_k$ for three distinct indices $1\leq i<j<k\leq5$. In a region where at least four $\lambda_i$ are distinct or exactly three $\lambda_i$ are distinct with multiplicities $(2,2,1)$, 
	\begin{equation}\label{Ass}
		A(r)=\sum_{i,j,k\in I_r} s^{ijk}\left(\left(s_1^{ijk}\right)^2+2s_2^{ijk}-\sigma_2\right)
		\left(2 s_1^{ijk}  s_2^{ijk} -3s_3^{ijk}+2\sigma_3\right)>0	
	\end{equation}
	for $r=1,2,\cdots,5$ at every point where the principal curvatures consist of at least four distinct values, or exactly three distinct values with multiplicities $(2,2,1)$.
%\end{itemize}
Then $M^5$ is isoparametric.
%Then $\operatorname{tr}(\mathcal{A}^5)$ is a constant, and furthermore, $\lambda_{1},\lambda_{2},\cdots,\lambda_{5}$ are constants.
\end{thm}

%$(\ref{Ass})$ is a technical assumption in the proof of Theorem \ref{main thm}. 
Now we give another characterization of $(\ref{Ass})$. 
%\begin{rem}\label{rem}
	%Then $(\ref{vec1})$ states that $v$ is always orthogonal to the vector $(1, -1, 1, -1)$, and $(\ref{vec2})$ states that the angle between $\omega$ and the vector $(1, -1, 1, -1)$ is always less than $ \frac{\pi }{2}$, with the specific value determined by the principal curvatures at the point $x\in M^5$. 
	 Lemma \ref{pvec} (see Section \ref{Bas})  implies that at each point of the hypersurface $M^5$, we have 
	 $$A(5)=  -\Bigl\langle  \left(z_4q_4, z_3q_3, z_2q_2, z_1q_1\right), \left(\lambda_{4}z_4q_4, \lambda_{3}z_3q_3, \lambda_{2}z_2q_2, \lambda_{1}z_1q_1\right) \Bigr\rangle. $$
%where $$v =,\,\omega = .$$  
	Therefore, %assumption 
	$A(5) > 0 $ is equivalent to the angle between vectors $(z_4q_4, z_3q_3, z_2q_2, z_1q_1)$ and $(\lambda_{4}z_4q_4, \lambda_{3}z_3q_3, \lambda_{2}z_2q_2, \lambda_{1}z_1q_1)$ lying between $\frac{\pi}{2}$ and $\pi$, where the two vectors are further constrained by $(\ref{vec1})$ and $(\ref{vec2})$, respectively. The remaining $A(i)$, $i\in I_5$, are similar.
%\end{rem}

%The above theorem holds without the assumption of nonnegative scalar curvature. After a detailed discussion of Assumption $(\ref{Ass})$, we have the following corollary. In this corollary, we provide some ranges for the principal curvatures within which Assumption $(\ref{Ass})$ holds. 
There are many configurations of principal curvatures that fit into assumption $\eqref{Ass}$. 
However, for the four principal curvature configurations in the following corollary, the scalar curvature $R$ may not  be nonnegative. %It would be of interest to find more configuration of principal curvatures  that satisfy $(\ref{Ass})$.%Therefore, in these cases, the assumptions $A(r)>0$ $(r=1,2,\cdots,5)$ are weaker than $R\geq 0.$
%This also shows there exist a large number of $\lambda_{i}$ $(i=1,2,\cdots,5)$ such that $A(r)>0$ for any $r$. 
%and by the continuity of $A(r)$ $(r=1,2,\cdots,5)$, it follows that in the neighborhood of these ranges, all $A(r)>0$ have the same sign.

%In the corollary, all principal curvature types we give satisfy the assumption 
%$\sigma_3\geq 0$.

\begin{cor}\label{cor}
	Let $M^5$ be a closed $5$-dimensional minimal hypersurface in the unit sphere $\mathbf{S}^6$ such that $S, f_3, f_4$ are constants.
  %Suppose that on a region of the manifold $ M $ with at least three distinct principal curvatures, the principal curvatures at each point belong to one of the four cases $(1'')$, $(2'')$, $(3'')$ or $(4'')$.
   At each point, suppose either there are at most three distinct principal curvatures with multiplicities not equal to $(2,2,1)$, or the configuration of principal curvatures belongs to one of the following four types: 
	\begin{itemize}
	
		\item[$(1)$] $\lambda_{1}<\lambda_{2}<\lambda_{3}=\lambda_{4}<0<\lambda_{5}$;
		\item[$(2)$] $\lambda_{1}=\lambda_{2}<\lambda_{3}<\lambda_{4}<\lambda_{5}$; 
		%\item[$(4'')$] $\lambda_{1}=\lambda_{2}<\lambda_{3}<\lambda_{4}<0<\lambda_{5}$. 
		\item[$(3)$] $\lambda_{1}<\lambda_{2}=\lambda_{3}<\lambda_{4}<0<\lambda_{5}$;
			\item[$(4)$] $\lambda_1<\lambda_{2}<\lambda_{3}<\lambda_{4}<0<\lambda_{5}$.
	\end{itemize}
	Then $M^5$ is isoparametric. 
\end{cor}

	Here we point out that in $(2)$, if $\lambda_{3}$ is nonnegative, $\sigma_3$ may be less than zero. In all other cases,  $\sigma_3$ is always nonnegative. 
In fact,  substituting $\lambda_{5}=-\lambda$ and $\lambda_{1}=\lambda_{2}$ into $\sigma_3$, we get 
	$$\sigma_3=-\left(2\lambda_{2}+\lambda_{3}\right){\lambda_{4}}^2-\left(4{\lambda_{2}}^2+4\lambda_{2}\lambda_{3}+{\lambda_{3}}^2\right)\lambda_{4}-2\lambda_{2}{\lambda_{3}}^2-4{\lambda_{2}}^2\lambda_{3}-2{\lambda_{2}}^3.$$
	
	In the situation $\lambda_{3}<\lambda_{4}<0<\lambda_{5}$, the conclusion holds trivially.
	
	In the situation $\lambda_{3}<0<\lambda_{4}<\lambda_{5}$, the conclusion follows from
	$-\left(2\lambda_{2}+\lambda_{3}\right)>0$ and its discriminant 
	$\Delta=-8{\lambda_{2}}^3\lambda_{3}-8{\lambda_{2}}^2{\lambda_{3}}^2+{\lambda_{3}}^4<0.$
	
  Furthermore,  we obtain a global rigidity result.

\begin{thm}\label{two}
	Let $M^5\subset\mathbf{S}^6$ be a closed minimal hypersurface with constant scalar curvature $R$ and constant $4$-th mean curvature $H_4$. 
Suppose there is a point with two distinct principal curvatures of multiplicities $(m_1,m_2)$. If $(m_1,m_2)=(1,4)$, 
	suppose in addition that the $3$-rd mean curvature $H_3$ is constant.
 Then $S=5$ and $M^5$ is the Clifford torus $\mathbb{S}^2(\sqrt\frac{2}{5})\times\mathbb{S}^3(\sqrt\frac{3}{5})$ or $\mathbb{S}^1(\sqrt\frac{1}{5})\times\mathbb{S}^4(\sqrt\frac{4}{5})$.
\end{thm}
%\begin{rem}
%	Here we point out that, as can be seen from the proof of Theorem \ref{two}, the hypothesis 
%	$f_3=\text{const.}$ is not needed for the case where the principal curvatures satisfy $\lambda_{1}=\lambda_{2}=\lambda_{3}<\lambda_{4}=\lambda_{5}$.
%\end{rem}

The rest of this paper is organized as follows. In Section \ref{Bas}, we give some preliminaries and establish some lemmas of this paper.  In Section \ref{simp}, we show our core lemma (Lemma \ref{corelem}). In Section \ref{Global thm}, we prove Theorem \ref{main thm}, Corollary \ref{cor} and Theorem \ref{two}.

\section{Preliminaries}\label{Bas}
In this section, we assume that $M^n$ is connected and oriented. Otherwise, we can discuss the situation on each connected component of $M^n$ or on the double covering of $M^n$.
\subsection{Basic Knowledge}
Let $f:M^n\rightarrow \mathbf{S}^{n+1}$ be an $n$-dimensional immersed hypersurface, and let $\{e_1,e_2,\cdots,e_{n+1}\}$ be an oriented local orthonormal frame field of $\mathbf{S}^{n+1}$ such that $e_1,e_2,\cdots,e_n$ are tangent to $M^n$. We use $\{\theta_i,i=1,2,\cdots,n\}$ and  $\{\omega_{ij}, 1\leq i,j\leq n\}$ to denote the dual $1$-forms and connection $1$-forms corresponding to $\{e_1,e_2,\cdots,e_n\}$, respectively. Then the structure equations of $M^n$ are given by:
$$
\begin{cases}
	d\theta_i=\sum_{j=1}^{n}\omega_{ij}\wedge\theta_j,\\
	d\omega_{ij}=\sum_{k=1}^{n}\omega_{ik}\wedge\omega_{kj}-R_{ij},
\end{cases}
$$
where $R_{ij}=\frac{1}{2}\sum_{k,l=1}^{n}R_{ijkl}\theta_k\wedge\theta_l$ denote the curvature $2$-forms of $M^n$.

Let $\operatorname{II}=\sum_{i,j=1}^{n}h_{ij}\theta_i\otimes\theta_j$ denote the second fundamental form, 
%$$\operatorname{II}=\sum_{i,j=1}^{n}h_{ij}\theta_i\otimes\theta_j$$
then the mean curvature is given by
$$H=\frac{1}{n}\sum_{i=1}^nh_{ii}.$$
Let $S=|\operatorname{II}|^2=\sum_{i,j=1}^{n}h_{ij}^2$ be the square length of the second fundamental form. Then the Gauss equation implies that 
% i.e.$$S=\sum_{i,j=1}^{n}h_{ij}^2.$$
%It is easy to see that for a minimal hypersurface $M$ in $\mathcal{S}^{n+1}$
\begin{equation}\label{gauss}
	R_{ijkl}=\delta_{ik}\delta_{jl}-\delta_{il}\delta_{jk}+h_{ik}h_{jl}-h_{il}h_{jk},
\end{equation}
$$R=n(n-1)+n^2H^2-S,$$
where $R$ is the scalar curvature of $M$. 

Define the covariant derivative $\nabla\operatorname{II}$ of $\operatorname{II}$ $($with component $h_{ijk}$$)$ by 
$$\sum_{m=1}^{n}h_{ijm}\theta_m=dh_{ij}+\sum_{m=1}^nh_{mj}\omega_{mi}+\sum_{m=1}^nh_{im}\omega_{mj}.$$
Then by Codazzi equation we have
\begin{equation}\label{coda}
h_{ijk}=h_{ikj}\ \mbox{for} \ i,j,k=1,2,\cdots,n.
\end{equation}
It implies immediately that $h_{ijk}$ is symmetric, and when $M$ is minimal, from \cite{PT1}, we know 
\begin{equation}\label{DeltaS}
	\frac{1}{2}\Delta S=(n-S)S+\sum_{i,j,k=1}^nh_{ijk}^2.
\end{equation}

Next we exterior differentiate the above formula and define $h_{ijkl}$ by 
$$\sum_{m=1}^{n}h_{ijkm}\theta_m=dh_{ijk}+\sum_{m=1}^{n}h_{mjk}\omega_{mi}+\sum_{m=1}^{n}h_{imk}\omega_{mj}+\sum_{m=1}^{n}h_{ijm}\omega_{mk},$$
and we define $f_3$ and $f_4$ as
$$f_3=\sum_{i,j,k=1}^nh_{ij}h_{jk}h_{ki},\quad f_4=\sum_{i,j,k,l=1}^{n}h_{ij}h_{jk}h_{kl}h_{li}.$$

For an arbitrary fixed point $x\in M^n$, we take an orthonormal frame such that $h_{ij}=\lambda_{i}\delta_{ij}$ at $x$, for all $i,j=1,2,\cdots,n$. %where $\lambda_{i}$ is principal curvature on $p$. 
Then at this point $x$, we have
$$f_3=\sum_{i=1}^{n}\lambda_{i}^3,\quad f_4=\sum_{i=1}^{n}\lambda_{i}^4,\quad H=\frac{1}{n}\sum_{i=1}^{n}\lambda_{i},\quad S=\sum_{i=1}^{n}\lambda_{i}^2,$$
and we define the smooth function 
$$ h=\sum_{i=1}^{n}\lambda_{i}^n.$$
%denote the second fundamental form and the mean curvature of the immersed manifold $M$ respectively. 

\begin{defn}\cite{dB}
	The combination $(U,\theta)$ is admissible if 
	\begin{itemize}
		\item $U$ is an open subset of $Y$, where $Y$ is given in  $(\ref{space})$;
		\item $\theta=(\theta_1,\theta_{2},\cdots,\theta_n)$ is a smooth orthonormal coframe field on $U$;
		\item $\theta_{1}\wedge\theta_{2}\wedge\cdots\wedge\theta_n=vol$ on $U$, where vol is the volume form of $U$;
		\item $\operatorname{II}=\sum_{i=1}^{n}\lambda_{i}\theta_i\otimes\theta_i.$
	\end{itemize}
\end{defn}
In this paper, we choose a proper system on $M^n$ such that $(U,\theta)$ is admissible. Then the connection form $\omega_{ij}$ on $U$ are uniquely determined and $h_{ij}=\lambda_{i}\delta_{ij}.$
In this admissible chart, we suppose $\lambda_{1}\leq \lambda_{2}\leq \cdots\leq \lambda_n$. Then we define an $(n-1)$-form $\Phi$ as follows, which is the key point of our proof.
\begin{equation}\label{Phi}
	\Phi=\sum_{\sigma}S(\sigma)(\lambda_{i_{n-1}}+\lambda_{i_n})\theta_{i_1}\wedge\theta_{i_2}\wedge\cdots\theta_{i_{n-1}}\wedge\omega_{i_{n-1}i_{n}},
\end{equation}
where $\sigma(1,\cdots,n)=(i_1,\cdots,i_n)$ is a permutation and $S(\sigma)$ is the sign of $\sigma$. By \cite{dB}, we know the $(n-1)$-form $\Phi$ is globally well-defined on $M^n$. In fact, every 
$\theta_{i_1}\wedge\theta_{i_2}\wedge\cdots\wedge\theta_{i_{n-1}}\wedge\omega_{i_{n-1}i_{n}}$
is well-defined.

Let $\sigma_r:\mathbb{R}^n\rightarrow\mathbb{R}$ be the elementary symmetric functions defined by 
$$\sigma_r(\lambda_{1},\cdots,\lambda_{n})=\sum_{i_1<i_2<\cdots<i_r}\lambda_{i_1}\lambda_{i_2}\cdots\lambda_{i_r}\ \mbox{for}\ 1\leq r\leq n,$$
and then define the $r$-th mean curvature by 
$$H_r=\frac{1}{C^r_n}\sigma_r.$$

%In the following contents of this subsection, we only consider $5$-dimensional manifolds.

Now we define the region $\Omega$ as follows.

\begin{equation*}\label{Om}
	\Omega=\left\{x\in M^n\middle|\sum_{i=1}^{n}\lambda_i^j(x)=c_j,\ \forall j=1,2,\cdots,n-1\,\text{and}\ \lambda_{1}(x)<\lambda_{2}(x)<\cdots<\lambda_{n}(x)\right\},
\end{equation*}
where $c_1,c_2,\cdots,c_{n-1}$ are constants. 

 The functions $\lambda_{i}$ $(i=1,2,\cdots,n)$ are smooth on $\Omega$.  
Thus we have
\begin{equation}\label{dlam}
	d\lambda_{i}=\sum_{j=1}^{n}\lambda_{ij}\theta_j,
\end{equation}
and $\lambda_{ij}$ are smooth functions on $\Omega$. In addition, we express connection coefficients of the connection form $\omega_{ij}$ as 
\begin{equation}\label{conform}
	\omega_{ij}=\sum_{k=1}^{n}\Gamma_{ijk}\theta_k,
\end{equation}
where $\Gamma_{ijk}=\omega_{ij}(e_k)$ for $i,j=1,2,\cdots,n$. From \cite{TWY}, we obtain
\begin{equation}\label{Gamma}
	h_{iik}=\lambda_{ik}\ \mbox{and}\ h_{ijk}=\left(\lambda_{i}-\lambda_{j}\right)\Gamma_{ijk}\ \mbox{for}\ i\neq j,
\end{equation}
and furthermore
\begin{equation}\label{lam1}
	\lambda_{ij}=(-1)^{n+1}\frac{h_j}{n}\cdot\frac{1}{\prod_{k=1; k\neq i}^{n} (\lambda_k - \lambda_i)},
\end{equation}
where $h_j$ is defined by 
\begin{equation}\label{dh}
	dh=\sum_{j=1}^{n}h_j\theta_j.
\end{equation}

%Next we introduce some notations and conclusions for hypersurfaces $M^n$  in the unit sphere $\mathbf{S}^{n+1}$ \cite{PT}, which will be needed in the proof of Theorem \ref{two}.

%For an arbitrary fixed point $x\in M$,
%let $\{e_1,\cdots,e_n\}$ and $\xi$ be orthonormal bases of
%$T_xM$ and $N_xM$, respectively.
%To better explain the assumptions
%$A(r)>0$ for all $r=1,2,\cdots 5$, 
In the following lemma, we provide another expression for 
$A(r)$.  We define $$v_k=\prod_{i,j=1;i<j;i,j\neq k}^4(\lambda_{i}-\lambda_{j}), v=\prod_{i,j=1;i<j}^{4}(\lambda_{i}-\lambda_{j})\ \mbox{and}\	t_k=(-1)^k\prod_{i=1;i\neq k}^{4}(\lambda_{i}-\lambda_k)$$
for $k=1,2,3,4.$
%	Let \begin{equation*}
%	\begin{aligned}
	%	t_k=(-1)^k\prod_{i=1;i\neq k}^{4}(\lambda_{i}-\lambda_k).\\
		%t_4=&\prod_{i=1}^{3}(\lambda_{i}-\lambda_{4}),\, t_3=-\prod_{i=1,2,4}(\lambda_{i}-\lambda_{3}),\,
		%t_2=\prod_{i=1,3,4}(\lambda_{i}-\lambda_{2}),\,
		%	t_1=-\prod_{i=2}^4(\lambda_{i}-\lambda_{1}).\\
		%t_1=&(\lambda_{1}-\lambda_{4})(\lambda_{2}-\lambda_{4})(\lambda_{3}-\lambda_{4}),\\
		%t_2=&(\lambda_{1}-\lambda_{3})(\lambda_{2}-\lambda_{3})(\lambda_{3}-\lambda_{4}),\\
		%t_3=&(\lambda_{1}-\lambda_{2})(\lambda_{2}-\lambda_{3})(\lambda_{2}-\lambda_{4}),\\
		%t_4=&(\lambda_{1}-\lambda_{2})(\lambda_{1}-\lambda_{3})(\lambda_{1}-\lambda_{4}).
%	\end{aligned}
%\end{equation*}
\begin{lem}\label{pvec}
	%Define  %$$z_4=\!\!\!\!\prod_{i,j=1,i<j}^3\!\!\!(\lambda_{i}-\lambda_{j}), z_3=\!\!\!\!\prod_{\substack{i<j\\i,j=1,2,4}}\!\!\!(\lambda_{i}-\lambda_{j}),z_2=\!\!\!\!\prod_{\substack{i<j\\i,j=1,3,4}}\!\!\!(\lambda_{i}-\lambda_{j}),$$
	%$$z_1=\!\!\!\!\prod_{\substack{i<j\\i,j=2,3,4}}\!\!\!(\lambda_{i}-\lambda_{j}),$$
Define $q_i$ $(i=1,2,3,4)$ by equation  $(\ref{qk})$, and $p_i$ $(i=1,2,3,4)$ by equation $(\ref{p})$ respectively. Then we have
	\begin{equation*}
		A(5)=-\left(\lambda_{4}v_4^2q_4^2+\lambda_{3}v_3^2q_3^2+\lambda_{2}v_2^2q_2^2+\lambda_{1}v_1^2q_1^2\right).
	\end{equation*}
	%where $q_i$ $(i\in I_5)$ are given in 
	Moreover, $A(i)$ $(i=1,2,3,4)$ is obtained by replacing $\lambda_{i}$ in $A(5)$ with $\lambda_{5}$. 
\end{lem}
\begin{proof}
Since
	\begin{equation}\label{vec1}
	v_4q_4-v_3q_3+v_2q_2-v_1q_1=0,
\end{equation}
from $(\ref{eql})$ and $(\ref{Al})$, we obtain
	\begin{equation*}
		\begin{aligned}
			A(5)=&v_4^2p_4q_4+v_3^2p_3q_3+v_2^2p_2q_2+v_1^2p_1q_1\\
			=&-v_4^2q_4\left(t_4+\lambda_{4}q_4\right)-v_3^2q_3\left(-t_3+\lambda_{3}q_3\right)-v_2^2q_2\left(t_2+\lambda_{2}q_2\right)-v_1^2q_1\left(-t_1+\lambda_{1}q_1\right)\\
			=&-v\left(v_4q_4-v_3q_3+v_2q_2-v_1q_1\right)-\lambda_{4}v_4^2q_4^2-\lambda_{3}v_3^2q_3^2-\lambda_{2}v_2^2q_2^2-\lambda_{1}v_1^2q_1^2\\
			=&-\left(\lambda_{4}v_4^2q_4^2+\lambda_{3}v_3^2q_3^2+\lambda_{2}v_2^2q_2^2+\lambda_{1}v_1^2q_1^2\right).
		\end{aligned}
	\end{equation*}
		
		Besides, by direct calculation, the following equation holds.
	\begin{equation}\label{vec2}
		-\lambda_{4}v_4q_4+\lambda_{3}v_3q_3-\lambda_{2}v_2q_2+\lambda_{1}v_1q_1=3v>0.
	\end{equation}
\end{proof}

\subsection{Preparations for the proof of Theorem \ref{main thm}}\label{pre}
%In this subsection, similar to \cite{TY},  in order to prove the main theorem, we classify the points on the manifold $M^5$ according to the principal curvatures. Furthermore, when the region of the manifold where all principal curvatures are  distinct is nonempty, we further divide this region using the range of a smooth function $h$ on $M$.

%In this subsection, we only consider the case $n=5$. 
Firstly, when $n=5$, we give the characteristic polynomial $F(x)$ %of the $(1,1)$-tensor field $\mathcal{A}$ 
and its relationship with the smooth function $h$.

\begin{equation*}
	F(x)=\prod_{i=1}^5(x-\lambda_{i})=x^5-\sigma_1x^4+\sigma_2x^3-\sigma_3x^2+\sigma_4x-\sigma_5.
\end{equation*}

From Newton's formula and the assumptions of Theorem \ref{main thm}, it follows that $\sigma_i$ $(i=1,2,3,4)$ are constants, and
$$\sigma_5=\frac{h}{5}+C_h,$$
where $C_h=-\frac{1}{6}Sf_3$ is a constant.

 Let
\begin{equation*}
	F_0(x)=x^5-\sigma_1x^4+\sigma_2x^3-\sigma_3x^2+\sigma_4x.
\end{equation*}
Obviously, $F_0(x)$ is a well-determined polynomial of degree $5$, and 
$$F(x)=F_0(x)-\frac{h}{5}-C_h.$$

Since $M^5$ is closed, we have that the range of $h$ is a closed interval, denoted by $\operatorname{Im}h=[a_0,b_0]$, $a_0\leq b_0$. 

From the definition of $\Omega$, it follows that $F(x)$ has $5$ distinct real roots $\lambda_{1}<\lambda_{2}<\lambda_{3}<\lambda_{4}<\lambda_{5}$ on $\Omega$. By Rolle's theorem, there exist $\tau_i$ $(i=1,2,3,4)$ lying between these roots such that
$$F_0'(\tau_i)=F'(\tau_i)=0 \ \mbox{for}\ i=1,2,3,4.$$
Therefore, $\tau_i$ $(i=1,2,3,4)$ are the extreme points of $F_0(x)$. Define 
$$b'=\min\{F_0(\tau_1),F_0(\tau_3)\}\ \mbox{and}\ a'=\max\{F_0(\tau_2),F_0(\tau_4)\},$$
as shown in Figure \ref{fig1}.

\begin{figure}[h]
	\centering
	\begin{tikzpicture}
		\begin{axis}[
			axis lines = none,
			xmin = -2.2, xmax = 2.2,
			ymin = -2.5, ymax = 2.5,
			domain = -2.2:2.2,
			samples = 100,
			width = 8cm,        % ?12cm???8cm
			height = 3.5cm,       % ?8cm???6cm
			clip = false
			]
			% ??????????
			\addplot[black, thick] {x^5 - 5*x^3 + 4*x};
			
			% ???????????
			\draw[black, dashed] (axis cs:-2.2,-1.4187) -- (axis cs:2.2,-1.4187)
			node[pos=0.98, above] {$F_0 = a'$};
			\draw[black, dashed] (axis cs:-2.2,1.4187) -- (axis cs:2.2,1.4187)
			node[pos=0.98, above] {$F_0 = b'$};
			
			% ????
			\addplot[mark=*, mark size=1.5pt, black, only marks] coordinates {
				(0.5439, 1.4187)
				(-0.5439, -1.4187)
			};
		\end{axis}
	\end{tikzpicture}
	\caption{ Function $F_0$.}
	\label{fig1}
\end{figure}
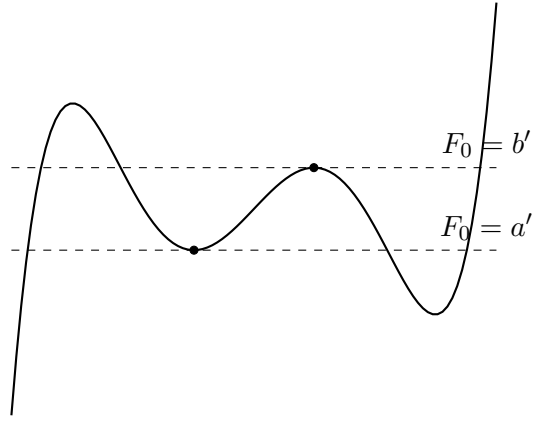

From the fact that $F(x)$ has $5$ distinct roots on $\Omega$, we have $b'>a'$. Furthermore, as described in \cite{TY}, for any $\xi\in [a_0,b_0]$, the equation 
$$F_0(x)-\frac{1}{5}\xi-C_h=0$$ has $5$ real roots. Let $b=5\left(b'-C_h\right)$ and $a=5\left(a'-C_h\right)$. We have
$$\operatorname{Im}h=[a_0,b_0]\subset[a,b].$$

For the case $a_0>a$, $b_0<b$, combined with the proof in this paper, Theorem \ref{main thm} can be proved by an argument as in \cite{TWY}. For the other cases, as described in \cite{TY}, it suffices to consider that $a_0=a$ and $b_0=b$. In this case, we have $\operatorname{Im}h=[a,b]$, and when $h=a$ (or $b$), $F_0(x)$ attains the maximum of all the local minimum values (or the minimum of all the local maximum values). 

Next, following \cite{dB}, we perform a region division on the manifold: $M^5=X\cup Y\cup Z$, where
\begin{equation}\label{space}
	\begin{aligned}
		&X:=\{x\in M^5: h(x)=a \}=h^{-1}(a),\\
		&Y:=\{x\in M^5: a<h(x)<b\},\\
		&Z:=\{x\in M^5: h(x)=b\}=h^{-1}(b).
	\end{aligned}
\end{equation}
Assume $Y\neq \emptyset$; otherwise the conclusion obviously holds. If $\Omega\neq \emptyset$, then $Y\subset \Omega$. Now, we introduce some notations: for $0<\epsilon<\frac{b-a}{2}$, write
\begin{equation*}
	\begin{aligned}
		&X_\epsilon:=\{x\in M^5: a<h(x)<a+\epsilon\},\\
		&Y_\epsilon:=\{x\in M^5: a+\epsilon\leq h(x)\leq b-\epsilon\},\\
		&Z_\epsilon:=\{x\in M^5: b-\epsilon<h(x)<b\},
	\end{aligned}
\end{equation*}
and then $Y=X_\epsilon\cup Y_\epsilon\cup Z_\epsilon$.

At the end of this subsection, we will present several lemmas used in the proof of Theorem \ref{main thm}. The following lemma gives the differential of the $(n-1)$-form $\Phi$.
\begin{lem}\label{dPhi} For a minimal hypersurface $M^n$ in the unit sphere $\mathbf{S}^{n+1}$, we have
	\begin{equation}\label{eqdPhi}
		\begin{aligned}
			d\Phi =& (-1)^n \,2\, \Bigg( \left(n-2\right)! \,f_3  + \sum_{\sigma} \frac{\left(\lambda_{i_{n-1} i_n} + \lambda_{i_ni_n}\right) \lambda_{i_{n-1} i_n}}{\lambda_{i_n} - \lambda_{i_{n-1}}} \\
	&\qquad\qquad +\sum_{\sigma} \lambda_{i_n} \sum_{k=1; k \neq i_{n-1}, i_n}^{n} \frac{\lambda_{k i_n} \lambda_{i_{n-1}i_n}}{\left(\lambda_k - \lambda_{i_n}\right)\left(\lambda_{i_{n-1}} - \lambda_{i_n}\right)}  \Bigg)\cdot \mathrm{vol}.
		\end{aligned}
	\end{equation}
	%where $vol=\theta_1\wedge\cdots\wedge\theta_n$ is the volume form of $M^n$.
\end{lem}
\begin{proof}
	The differential of $\Phi$ can be calculated by parts as follows:
	\begin{equation}\label{comdPhi}
		\begin{aligned}
			d\Phi=&\sum_{\sigma}S(\sigma)\,d\left(\lambda_{i_{n-1}}+\lambda_{i_n}\right)\wedge\theta_{i_1}\wedge\theta_{i_2}\wedge\cdots\wedge\theta_{i_{n-2}}\wedge\omega_{i_{n-1}i_{n}}\\
			&+\sum_{\sigma}S(\sigma)\left(\lambda_{i_{n-1}}+\lambda_{i_n}\right)d\left(\theta_{i_1}\wedge\theta_{i_2}\wedge\cdots\wedge\theta_{i_{n-2}}\wedge\omega_{i_{n-1}i_{n}}\right)\\
			:=&\phi+\psi.
		\end{aligned}
	\end{equation}
	
	Using $(\ref{dlam}-\ref{Gamma})$, we have
	\begin{equation*}
		\begin{aligned}
			\phi_1&:=\sum_{\sigma}S(\sigma)\,d\lambda_{i_{n-1}}\wedge\theta_{i_1}\wedge\theta_{i_2}\wedge\cdots\wedge\theta_{i_{n-2}}\wedge\left(\sum_{k=1}^n\Gamma_{i_{n-1}i_nk}\,\theta_k\right)\\
			&=\sum_{\sigma}S(\sigma)\left(\sum_{j=1}^{n}\lambda_{i_{n-1}j}\,\theta_j\right)\wedge\theta_{i_1}\wedge\theta_{i_2}\wedge\cdots\wedge\theta_{i_{n-2}}\wedge\left(\sum_{k=1}^n\,\Gamma_{i_{n-1}i_nk}\,\theta_k\right)\\
			&=\sum_{\sigma}S\left(\sigma\right)\Bigg(\lambda_{i_{n-1}i_{n-1}}\,\theta_{i_{n-1}}\wedge\theta_{i_1}\wedge\theta_{i_2}\wedge\cdots\wedge\theta_{i_{n-2}}\wedge\Gamma_{i_{n-1}i_ni_n}\,\theta_{i_n}\\
			&\qquad\qquad\qquad+\lambda_{i_{n-1}i_n}\,\theta_{i_n}\wedge\theta_{i_1}\wedge\theta_{i_2}\wedge\cdots\wedge\theta_{i_{n-2}}\wedge\Gamma_{i_{n-1}i_ni_{n-1}}\,\theta_{i_{n-1}}\Bigg)\\
			&=(-1)^n\,\sum_{\sigma}\Bigg(\frac{\lambda_{i_{n-1}i_{n-1}}\,\lambda_{i_ni_{n-1} }-\lambda_{i_{n-1} i_n}\,\lambda_{i_{n-1} i_n}}{\lambda_{i_{n-1}}-\lambda_{i_n}}\Bigg)\cdot \mathrm{vol}.
		\end{aligned}
	\end{equation*}
	Similarly,
	$$\begin{aligned}
		\phi_2:=&\sum_{\sigma}S(\sigma)\,d\lambda_{i_{n}}\wedge\theta_{i_1}\wedge\theta_{i_2}\wedge\cdots\wedge\theta_{i_{n-2}}\wedge\omega_{i_{n-1}i_{n}}\\
		=&(-1)^n\,\sum_{\sigma}\Bigg(\frac{\lambda_{i_{n}i_{n-1}}\lambda_{i_ni_{n-1}}-\lambda_{i_ni_n}\lambda_{i_{n-1}i_n}}{\lambda_{i_{n-1}}-\lambda_{i_n}}\Bigg)\cdot \mathrm{vol}.
		\end{aligned}$$
	Thus,
	\begin{equation}\label{phi}
		\phi=\phi_1+\phi_2=(-1)^n\,2\,\sum_{\sigma}\left(\frac{\left(\lambda_{i_{n-1} i_n}+\lambda_{i_ni_n}\right)\lambda_{i_{n-1} i_n}}{\lambda_{i_n}-\lambda_{i_{n-1}}}\right)\cdot \mathrm{vol}.
		\end{equation}
	%For the second term,we know that
	
	%To compute $\psi$, we decompose $\psi=\psi_1+\psi_2-\psi_3$.
	According to \cite{TWY}, we obtain
	\begin{equation*}
		\begin{aligned}
			\psi_1:=&\sum_{\sigma}S(\sigma)\left(\lambda_{i_{n-1}}+\lambda_{i_n}\right)d\left(\theta_{i_1}\wedge\theta_{i_2}\wedge\cdots\wedge\theta_{i_{n-2}}\right)\wedge\omega_{i_{n-1}i_{n}}\\
			=& (-1)^n\, \sum_{\sigma}\left(\lambda_{i_{n-1}}+\lambda_{i_n}\right) \sum_{k=1; k \neq i_{n-1}, i_n}^{n} \Bigg( -\frac{\lambda_{k i_{n-1}}\lambda_{i_n i_{n-1}}}{(\lambda_k - \lambda_{i_{n-1}})(\lambda_{i_{n-1}} - \lambda_{i_n})}\\
			&\qquad\qquad\quad+ \frac{h_{k i_{n-1} i_n}^2}{(\lambda_k - \lambda_{i_{n-1}})(\lambda_{i_{n-1}} - \lambda_{i_n})} + \frac{\lambda_{k i_n}\lambda_{i_{n-1} i_n}}{(\lambda_k - \lambda_{i_n})(\lambda_{i_{n-1}} - \lambda_{i_n})} \\
			&\qquad\qquad\qquad	\qquad\qquad\qquad\qquad \left. - \frac{h^2_{k i_{n-1} i_n}}{(\lambda_k - \lambda_{i_n})(\lambda_{i_{n-1}} - \lambda_{i_n})} \right)\cdot \mathrm{vol},
		\end{aligned}
	\end{equation*}
	and since
	$$\sum_{\sigma}\left(\lambda_{i_{n-1}}+\lambda_{i_n}\right)\sum_{k=1;k\neq i_{n-1},i_n}^{n}\frac{h_{ki_{n-1}i_n}^2}{(\lambda_k-\lambda_{i_{n-1}})(\lambda_{i_{n-1}}-\lambda_{i_n})}=0,$$
	we get
	\begin{equation}\label{psi1}
		\psi_1=(-1)^n\,2\,\sum_{\sigma}\left(\lambda_{i_{n-1}}+\lambda_{i_n}\right)\sum_{k=1;k\neq i_{n-1},i_n}^{n}\frac{\lambda_{ki_n}\lambda_{i_{n-1}i_n}}{(\lambda_k-\lambda_{i_n})(\lambda_{i_{n-1}}-\lambda_{i_n})}\cdot \mathrm{vol}.
	\end{equation}
	
	It also follows from \cite{TWY} that
	\begin{equation}\label{psi2}
		\begin{aligned}
			\psi_{2}&=(-1)^n\sum_{\sigma}S(\sigma)\left(\lambda_{i_{n-1}}+\lambda_{i_n}\right)\theta_{i_1}\wedge\theta_{i_2}\wedge\cdots\wedge\theta_{i_{n-2}}\\
			&	\qquad\qquad\qquad	\qquad\qquad\qquad	\qquad\wedge \left(\sum_{k=1;k\neq i_{n-1},i_n}^{n}\omega_{i_{n-1}k}\wedge\omega_{ki_n}\right)\\
			&=(-1)^n\sum_{\sigma}\left(\lambda_{i_{n-1}}+\lambda_{i_n}\right)\sum_{k=1;k\neq i_{n-1},i_n}^{n}\left(\frac{\lambda_{i_{n-1}k}\,\lambda_{i_nk}}{(\lambda_k-\lambda_{i_n})(\lambda_{i_{n-1}}-\lambda_{k})}\right.\\
			&	\qquad\qquad\qquad\qquad\qquad\qquad	\qquad\quad  \left.-\frac{h_{ki_{n-1}i_n}^2}{(\lambda_k-\lambda_{i_n})(\lambda_{i_{n-1}}-\lambda_{k})}\right)\cdot \mathrm{vol}\\
			&=(-1)^n\sum_{\sigma}\left(\lambda_{i_{n-1}}+\lambda_{i_n}\right)\sum_{k=1;k\neq i_{n-1},i_n}^{n}\frac{\lambda_{i_{n-1}k}\,\lambda_{i_nk}}{(\lambda_k-\lambda_{i_n})(\lambda_{i_{n-1}}-\lambda_{k})}\cdot \mathrm{vol},
		\end{aligned}
	\end{equation}
	where the last equation follows from the following equation:
	$$\sum_{\sigma}\left(\lambda_{i_{n-1}}+\lambda_{i_n}\right)\sum_{k=1;k\neq i_{n-1},i_n}^{n}\frac{h_{ki_{n-1}i_n}^2}{(\lambda_k-\lambda_{i_n})(\lambda_{i_{n-1}}-\lambda_k)}=0,$$
	whose proof can refer to \cite{CL}.
	
	Moreover, from $(\ref{gauss})$, when $H=0$ we have
	$$\sum_{i,j=1;i\neq j}^{n}\left(\lambda_{i}+\lambda_j\right)R_{ijij}=-2f_3.$$
	Therefore, we obtain
	\begin{equation}\label{psi3}
		\begin{aligned}
			\psi_{3}&=(-1)^n\,\sum_{\sigma}S(\sigma)\left(\lambda_{i_{n-1}}+\lambda_{i_n}\right)\theta_{i_1}\wedge\theta_{i_2}\wedge\cdots\wedge\theta_{i_{n-2}}\wedge R_{i_{n-1}i_n}\\
			&=(-1)^n\,\sum_{\sigma}\left(\lambda_{i_{n-1}}+\lambda_{i_n}\right)R_{i_{n-1}i_ni_{n-1}i_n}\,\theta_{1}\wedge\theta_{2}\wedge\cdots\wedge\theta_{n-1}\wedge \theta_{n}\\
			&=2\,(-1)^{n+1}\,(n-2)!\,f_3\cdot \mathrm{vol}.
		\end{aligned}
	\end{equation}
	%Where the last equal sign above is because when $H=0$, there have 
	
	By $(\ref{psi1}-\ref{psi3})$, we have
	\begin{equation}\label{psi}
		\begin{aligned}
			\psi&=\psi_1+\psi_2-\psi_3\\
			&=\sum_{\sigma}S(\sigma)\left(\lambda_{i_{n-1}}+\lambda_{i_n}\right)d\left(\theta_{i_1}\wedge\theta_{i_2}\wedge\cdots\wedge\theta_{i_{n-2}}\right)\wedge\omega_{i_{n-1}i_n}\\
			&\qquad\qquad\qquad\qquad+(-1)^n\,\theta_{i_1}\wedge\theta_{i_2}\wedge\cdots\wedge\theta_{i_{n-2}}\\
			&\qquad\qquad\qquad\qquad\qquad\wedge\left(\sum_{k=1;k\neq i_{n-1},i_n}^{n}\omega_{i_{n-1}k}\wedge\omega_{ki_n}-R_{i_{n-1}i_n}\right)\\
			&=(-1)^n\left(2\,\sum_{\sigma}\left(\lambda_{i_{n-1}}+\lambda_{i_n}\right)\sum_{k=1;k\neq i_{n-1},i_n}^n\frac{\lambda_{ki_n}\lambda_{i_{n-1}i_n}}{(\lambda_k-\lambda_{i_n})(\lambda_{i_{n-1}}-\lambda_{i_n})}\right.\\
			&	\qquad\qquad-\sum_{\sigma}\left(\lambda_{i_{n-1}}+\lambda_{i_n}\right)\sum_{k=1;k\neq i_{n-1},i_n}^{n}\frac{\lambda_{i_{n-1}k}\lambda_{i_nk}}{(\lambda_k-\lambda_{i_n})(\lambda_{k}-\lambda_{i_{n-1}})}\\
			&\qquad\qquad\qquad\qquad\qquad\qquad\qquad\qquad\qquad\quad+2\,(n-2)!\,f_3\Bigg)\cdot \mathrm{vol}\\
			&=(-1)^n\,2\,\sum_{\sigma}\lambda_{i_n}\sum_{k=1;k\neq i_{n-1},i_n}^n\frac{\lambda_{ki_n}\lambda_{i_{n-1}i_n}}{(\lambda_k-\lambda_{i_n})(\lambda_{i_{n-1}}-\lambda_{i_n})}\cdot \mathrm{vol}\\
			&\qquad\qquad\qquad\qquad\qquad\qquad\qquad\qquad\quad+(-1)^n\,2\,(n-2)!\,f_3\cdot \mathrm{vol}.\\
		\end{aligned}
	\end{equation}
	%where the last equality follows from the fact that 
	The last equality above holds because
	$$\begin{aligned}
		&2\sum_{\sigma}\lambda_{i_{n-1}}\sum_{k=1;k\neq i_{n-1},i_n}^n\frac{\lambda_{ki_n}\lambda_{i_{n-1}i_n}}{(\lambda_k-\lambda_{i_n})(\lambda_{i_{n-1}}-\lambda_{i_n})}\\
		&\qquad\qquad\qquad=\sum_{\sigma}\left(\lambda_{i_{n-1}}+\lambda_{i_n}\right)\sum_{k=1;k\neq i_{n-1},i_n}^n
		\frac{\lambda_{i_{n-1}k}\lambda_{i_nk}}{(\lambda_k-\lambda_{i_n})(\lambda_{i_{n-1}}-\lambda_{i_n})}.
		\end{aligned}$$
	%From the above discussion, $d\Phi =\Romannum{1}+\Romannum{2}$ is exactly the form appearing in the proposition.
	
	Finally, combining $(\ref{comdPhi}, \ref{phi},\ref{psi})$, we complete the proof of the lemma.
\end{proof}
The following lemma gives the $5$-form $dh\wedge\Phi$. Since the lemma can be obtained by direct calculation, we only provide a brief proof here. We refer the reader to \cite{HXZ} for detailed computations.
\begin{lem}\label{lemu}
For a hypersurface $M^5$ in the unit sphere $\mathbf{S}^6$, we have
	\begin{equation}\label{u}
		dh\wedge\Phi=\sum_{i=1}^{5}u_ih_i^2\cdot vol,%=\sum_{i=1}^{5}-\frac{12}{5}v_ih_i^2\cdot vol,
	\end{equation}
where% have 
	\begin{equation}\label{v}
	u_i=-\frac{12}{5}\sum_{k=1;k\neq i }^{5}\frac{\lambda_{k}+\lambda_{i}}{\left(\lambda_{k}-\lambda_{i}\right)^2\prod_{j=1;j\neq k, i}^5\left(\lambda_{j}-\lambda_{k}\right)}\ \mbox{for}\  i=1,2,\cdots,5.
	\end{equation}
\end{lem}
%for $ i=1,2,\cdots,5$.

\begin{proof}
	We express $\Phi$ as follows:
	$$\Phi=12\sum_{i,j=1;i<j}^{5}(\lambda_{i}+\lambda_{j})\theta_k\wedge\theta_l\wedge\theta_m\wedge\omega_{ij},$$
where $\{k,l,m\}=\{1,2,\cdots,5\}\setminus\{i,j\}$, and we choose the ordering of $k,l,m$ so that the permutation $\sigma'(1,2,3,4,5)=(k,l,m,i,j)$ satisfies $S(\sigma')=1$.% for.% where 
	%For the case $i_4=1,i_5=2$, we can easily get that the term containing $\omega_{12}$ of $\Phi$ is $$12\left(\lambda_{1}+\lambda_2\right)\theta_3\wedge\theta_4\wedge\theta_5\wedge\omega_{12}.$$
	
	According to $(\ref{conform}, \ref{dh})$, through direct calculation, we obtain
	$$dh\wedge\Phi=12\sum_{i,j=1;i<j}^{5}(\lambda_{i}+\lambda_{j})\left(-h_i\Gamma_{ijj}+h_j\Gamma_{iji}\right)\cdot \mathrm{vol}.$$
	% that the term containing $\lambda_{1}+\lambda_{2}$ in the coefficient of $dh\wedge\Phi$ is 
	%$$-12\left(\lambda_{1}+\lambda_{2}\right)\left(h_1\Gamma_{122}+h_2\Gamma_{121}\right)\cdot vol.$$
	%Hence
%	$$dh\wedge\Phi=-12\sum_{i,j=1;i<j}^{5}\left(\lambda_{i}+\lambda_{j}\right)\left(h_i\Gamma_{ijj}+h_j\Gamma_{iji}\right).$$
	Replacing $\Gamma_{ijj}$ and $\Gamma_{iji}$ by $(\ref{coda}, \ref{Gamma}, \ref{lam1})$, we obtain the conclusion.
	\end{proof}
Using Lemma \ref{lemu}, we now give a lemma on the boundedness of $u_i$.
\begin{lem}\label{bound}
	There exists a constant $C>0$ depending only on $c_1, c_2, c_3, c_4$ such that
	\begin{align*}
		u_i \leq C\ \mbox{on}\ X_\epsilon\ \mbox{and}\ u_i \geq -C\ \mbox{on }\ Z_\epsilon.
	\end{align*}
\end{lem}
\begin{proof}
	Define $$h^{-1}(b)=(\beta_1,\beta_2,\beta_3,\beta_4,\beta_5) \ \mbox{and}\
	h^{-1}(a)=(\alpha_1,\alpha_2,\alpha_3,\alpha_4,\alpha_5).$$ 
	Then Figure \ref{fig1} shows that the multiplicities of $\alpha_{i}$ or $\beta_i$ $(i=1,2,\cdots,5)$ are at most $2$ as $h\rightarrow a$ or $h\rightarrow b$. However, 
	%a direct computation shows that
	%at e if there exist distinct $i<j<k<l\in \{1,2,\cdots,5\}$ such that $\lambda_i = \lambda_j\neq\lambda_k = \lambda_l$ and both of them are not equal to $-2\lambda_{i}-2\lambda_{k}$,
at a point where there are three distinct principal curvatures with multiplicities $(2, 2, 1)$, we have $ A(r) = 0 $ for any $r = 1, 2,\cdots, 5$,
	 %then $A(r) = 0$ for any $r\in \{1,2,\cdots,5\}$, 
	 which contradicts $(\ref{Ass})$.  Moreover, if $\lambda_{1}<\lambda_{2}<\lambda_{3}<\lambda_{4}=\lambda_{5}$, then we have 
	$$A(3)=-2\lambda_{5}\prod_{i,j=1,2,5;i<j}{\left(\lambda_{i}-\lambda_{j}\right)}^2\mathcal{P}_{A(3)}^2<0,$$	
	which also contradicts assumption $(\ref{Ass})$. Here $\mathcal{P}_{A(3)}$ is a homogeneous polynomial of degree two in $\lambda_{i}$ $(i=1,2,5)$. 
	Consequently, such points cannot exist on the hypersurface $M^5$. Hence, only the following cases are possible.
	%{\left(2{\alpha_{1}}^2+5\alpha_{1}\alpha_{2}+6\alpha_{1}\alpha_{5}+2{\alpha_{2}}^2+6\alpha_{2}\alpha_{5}+4{\alpha_{5}}^2\right)}
	
	By Figure \ref{fig1}, when $h\rightarrow a$, %the following possible cases arise.
	we have $\alpha_{1}<\alpha_2=\alpha_{3}<\alpha_{4}<\alpha_{5}$. 
	%In addition, 
	From $(\ref{v})$, we can write $u_1$ as
	\begin{equation*}
		\begin{aligned}
			u_1=&-\frac{12}{5}\Bigg(	\frac{\mathcal{P}_{v_1}}{{\left(\lambda_{1}-\lambda_{2}\right)}^2{\left(\lambda_{1}-\lambda_{3}\right)}^2\left(\lambda_{2}-\lambda_{4}\right)\left(\lambda_{2}-\lambda_{5}\right)\left(\lambda_{3}-\lambda_{4}\right)\left(\lambda_{3}-\lambda_{5}\right)}\\
			&+\frac{\lambda_{1}+\lambda_{4}}{(\lambda_{1}-\lambda_{4})^2\prod_{i=2,3,5}(\lambda_{i}-\lambda_{4})}+\frac{\lambda_{1}+\lambda_{5}}{(\lambda_{1}-\lambda_{5})^2\prod_{i=2,3,4}(\lambda_{i}-\lambda_{5})}\Bigg)
			%+\frac{V_1}{\prod\limits_{i}(\alpha_2-\alpha_i)}		
		\end{aligned}
	\end{equation*}
	where $\mathcal{P}_{v_1}$ is a polynomial of $\lambda_{i}$ $(i=1,2,\cdots,5)$, thus $u_1$ is bounded. Similarly, $u_4$ and $u_5$ are bounded. For $u_2,u_3$, from $(\ref{v})$ we know when $h\rightarrow a$, they tend to $-\infty$  if $\alpha_{2}=\alpha_{3}<0$, which can be obtained by the assumption $A(5)>0$. More precisely, it follows from 
	$$A(5)=-2\alpha_{3}\prod_{i,j=1,3,4;i<j}{\left(\alpha_{i}-\alpha_{j}\right)}^2\mathcal{P}_{A(5)}^2>0$$
	that $\alpha_{2}=\alpha_{3}<0$. Here $\mathcal{P}_{A(5)}$ is a homogeneous polynomial of degree two in $\alpha_{i}$ $(i=1,3,4)$. Thus, $u_2$ and $u_3$ have upper bounds.
	
	By Figure \ref{fig1}, 
	When $h\rightarrow b$, we have
	$$\beta_{1}=\beta_2<\beta_{3}<\beta_{4}<\beta_{5}\ \mbox{or}\  \beta_{1}<\beta_2<\beta_{3}=\beta_{4}<\beta_{5}.$$ Then, similar
	to the discussion for $h\rightarrow a$, we have that %$v_3,v_4,v_5$ $(\mbox{or}\ v_1,v_2,v_5)$, and
	  $u_3,u_4,u_5$ $(\mbox{or}\ u_1,u_2,u_5)$ are bounded. Moreover, from $A(5)>0$, we obtain $\beta_{1}=\beta_2<0$ $(\mbox{or}\ \beta_{3}=\beta_4<0)$, and therefore, $u_1,u_2\rightarrow +\infty$ $(\mbox{or}\ u_3,u_4\rightarrow +\infty)$.
	  
	  The explicit expressions of the polynomials $\mathcal{P}_{A(3)}$, $\mathcal{P}_{v_1}$, and $\mathcal{P}_{A(5)}$ mentioned above are provided in Appendix \ref{A}.
\end{proof}
	%he proof of the theorem also requires the following lemma. 
	The following Lemma \ref{lemdB} was first proved by de Almeida and Brito \cite{dB} for the case $n=3$, and then Tang and Yan \cite{TY} pointed out that it also holds for arbitrary dimension $n$.
	\begin{lem}\label{lemdB}
	Suppose $u:M^5\rightarrow\mathbb{R}$ is smooth and $m=\min_{M^5}u$. If $D_\epsilon=u^{-1}([m,m+\epsilon])$, then
	$$\lim_{\epsilon\to 0} \int_{D_\epsilon}|\Delta u|\cdot\mathrm{vol}=0.$$
	In particular,
	$$\lim_{\epsilon\to 0}\int_{M^5-Y_\epsilon}|\Delta h|\cdot\mathrm{vol}=0,\ \mbox{if}\ X\cup Z\neq \emptyset.$$
\end{lem}

\section{Simplification of $d\Phi$}\label{simp}
In this section, we give a more explicit expression for $d\Phi$ when $n=5$, which will be used in the proof of Theorem \ref{main thm}.
\begin{lem}\label{corelem}
 For a minimal hypersurface $M^5$ in the unit sphere $\mathbf{S}^{6}$, we have
		\begin{equation}\label{key form}
		\begin{aligned}
			d\Phi =-12\left(3\,\sigma_3 +\frac{1}{25\prod_{i,j=1;i<j}^5(\lambda_{i}-\lambda_{j})^2}\sum_{r=1}^5A(r)h_r^2\right)\cdot \mathrm{vol}.
		\end{aligned}
	\end{equation}
\end{lem}

\begin{proof}
Let 
	\begin{equation}\label{letL}
	\begin{aligned}
		d\Phi =& (-1)^n \,2 \,\Big( \left(n-2\right)!\, f_3 + \mathcal{L}\Big)\cdot \mathrm{vol}\\
		=&(-1)^n \,2\, \left( \left(n-2\right)!\, f_3 + \frac{(n-2)!}{n^2}\sum_{r=1}^nL(r)h_r^2\right)\cdot \mathrm{vol}.
	\end{aligned}
\end{equation}
%$\mathcal{L}$ denote the summation term inside the parentheses in $(\ref{dPhi})$. 
Then it follows from $(\ref{eqdPhi})$ and $(\ref{lam1})$ that
\begin{equation}\label{L}
	\begin{aligned}
		\mathcal{L}=&(n-2)!\sum_{\substack{p,q,r=1 \\ p,q,r \text{ are distinct}}}^{n}\lambda_{r}\frac{\lambda_{qr}\lambda_{pr}}{(\lambda_{q}-\lambda_{r})(\lambda_{p}-\lambda_{r})}+(n-2)!\sum_{\substack{p,r=1 \\ p\neq r}}^n\frac{(\lambda_{pr}+\lambda_{rr})\lambda_{pr}}{(\lambda_{r}-\lambda_{p})}\\
		=&\frac{(n-2)!}{n^2}\sum_{r=1}^n\left(\sum\limits_{p=1; p\neq r}^{n}
		\frac{1}{(\lambda_r - \lambda_p) \left( \prod_{k=1; k\neq p}^{n} (\lambda_k - \lambda_p)\right)^2}	\right.\\
			&\qquad\qquad+ \sum\limits_{p=1; p\neq r}^{n}\frac{1}{(\lambda_r - \lambda_p) 
				\prod_{k=1;k\neq p}^{n} (\lambda_k - \lambda_p)
				\prod_{l=1;l\neq r}^{n} (\lambda_l - \lambda_r)}\\
		&\left.+\sum\limits_{\substack{p,q=1\\ p\neq q;p,q\neq r}}^{n} 
		\frac{\lambda_r}{(\lambda_r-\lambda_p)(\lambda_r-\lambda_q)\prod_{k=1;k\neq p}^{n} (\lambda_k - \lambda_p)
			\prod_{l=1;l\neq q}^{n} (\lambda_l - \lambda_q)}   
		\right)h_r^2.
	\end{aligned}
\end{equation}

%Now we define $L(r)$ as follows:
%$$\mathcal{L}=\frac{(n-2)!}{n^2}\sum_{r=1}^nL(r)h_r^2.$$
Next, we only consider the case $n=5$. %in order to express $L(r)$ for $r=1,2,\cdots,5$, 
We take $r=5$ as a representative example.
Firstly, we introduce some notations.
Let
\begin{equation}\label{xik}
	\xi_k=\sum_{i,j=1;i<j;i,j\neq k }^{4}\lambda_{i}\lambda_{j}+\lambda_k\lambda_{5}-(\lambda_k+\lambda_{5})\sum_{i=1;i\neq k}^{4}\lambda_{i}+\lambda_{k}^2+\lambda_{5}^2
\end{equation}
and 
\begin{equation}\label{ak}
\begin{aligned}
	a_k=&-\prod_{i,j=1;i<j;i,j\neq k}^4(\lambda_{i}-\lambda_{j})^2\prod_{i=1;i\neq k}^4(\lambda_{i}-\lambda_{5})\xi_k
\end{aligned}
\end{equation}
for $k=1,2,3,4.$
%where %$$\xi_1=\lambda_{2}\lambda_{3}+\lambda_{1}\lambda_{5}+\lambda_{2}\lambda_{4}+\lambda_{3}\lambda_{4}-\left(\lambda_{1}+\lambda_{5}\right)\left(\lambda_{2}+\lambda_{3}+\lambda_{4}\right)+{\lambda_{1}}^2+{\lambda_{5}}^2$$
%and $\xi_k(k=2,3,4)$ is obtained by swapping $\lambda_1$ and $\lambda_{k}$ in $\xi_1$.

Write
\begin{equation}\label{akl}
\begin{aligned}
	a^{(kl)}=(-1)^{k+l+1}\prod_{\substack{i,j=1;i<j\\(i,j)\neq (k,l),(k,5),(l,5)}}^5(\lambda_{i}-\lambda_{j})(\lambda_{m}-\lambda_{n})(\lambda_{m}-\lambda_{5})(\lambda_{n}-\lambda_{5})
\end{aligned}
\end{equation}
for $k<l$, $m<n$ and $\{k,l,m,n\}=\{1,2,3,4\}$.

Then we define
\begin{equation}\label{mA}
\mathcal{A}=-2\lambda_{5}\sum_{k,l=1;k<l}^{4}a^{(kl)}+\sum_{k=1}^{4}a_k,
\end{equation}
and by putting $L(5)$ over a common denominator, we obtain
\begin{equation}\label{eqLl}
L(5)=\frac{\mathcal{A}}{	\prod_{i,j=1;i<j}^5(\lambda_{i}-\lambda_{j})^2}.
\end{equation}
The expanded form of $L(5)$ is given in Appendix \ref{A1}.
% it is determined that 
%\begin{equation}\label{eqLl}
%	\prod_{i,j=1;i<j}^4(\lambda_{i}-\lambda_{j})^2L(5)=\prod_{i=1}^4(\lambda_{i}-\lambda_{5})^2l.
%\end{equation}
%$$6L(5)=\frac{(\lambda_{1}-\lambda_{5})^2(\lambda_{2}-\lambda_{5})^2(\lambda_{3}-\lambda_{5})^2(\lambda_{4}-\lambda_{5})^2}{(\lambda_{2}-\lambda_{1})^2(\lambda_{4}-\lambda_{1})^2(\lambda_{3}-\lambda_{1})^2(\lambda_{4}-\lambda_{2})^2(\lambda_{3}-\lambda_{2})^2(\lambda_{4}-\lambda_{3})^2}l.$$
%It follows that 
%\begin{equation}
%	L(5) > 0 \, \iff  \, l > 0.
%\end{equation}

In the following, we show that $\mathcal{A}=A(5)$. Let $\lambda=\lambda_{1}+\lambda_{2}+\lambda_{3}+\lambda_{4}$. Since $H=0$, substituting $\lambda_{5}$ for $-\lambda$ in $\mathcal{A}$, we obtain
$$\begin{aligned}
	\mathcal{A} = &2\lambda\sum_{k,l=1;k<l}^4(-1)^{k+l+1}b^{(kl)}%_{new}
	-\sum_{k=1}^{4}q_k\prod_{i,j=1;i<j;i,j\neq k}^4(\lambda_{i}-\lambda_{j})^2\prod_{i=1;i\neq k}^{4}(\lambda_{i}+\lambda),
\end{aligned}
$$
where
$$\begin{aligned}
	b^{(kl)}=\prod_{\substack{i,j=1;i<j;(i,j)\neq (k,l)}}^4(\lambda_{i}-\lambda_{j})(\lambda_{m}-\lambda_{n})(\lambda_{m}+\lambda)^2(\lambda_{n}+\lambda)^2
\end{aligned}
$$
for $k<l,m<n$, $\{k,l,m,n\}=\{1,2,3,4\}$,
and
\begin{equation}\label{qk}
	\begin{aligned}
		q_k &= \lambda_{k}^2 +2\sum_{i=1;i\neq k}^{4}\lambda_{i}^2+\sum_{i=1;i\neq k}^{4}\lambda_{i}\lambda_{k}+5\sum_{i,j=1;i<j;i,j\neq k}^{4}\lambda_{i}\lambda_{j}
	\end{aligned}
\end{equation}
for $ k=1,2,3,4.$ Here $b^{(kl)}$ and $q_k$ correspond to $a^{(kl)}$ and $\xi_k$ in (\ref{mA}), respectively.
%for $k=1,2,\cdots,4$.

%We now proceed to further simplify $l$.
Notice that 
$$(\lambda+\lambda_{i})(\lambda+\lambda_{j})=q_4-(\lambda_{m}-\lambda_{n})(2\lambda_{i}+2\lambda_{j}+\lambda_{k})$$
for $\{i,j,k\}=\{1,2,3\}$ and $\{i,j,m,n\}=\{1,2,3,4\}$.
Substituting these three equalities into $b^{(kl)}$ $(k<l,k,l=1,2,3,4)$, we obtain
\begin{equation}\label{lnew1}
	\begin{aligned}
		&b^{(34)}-b^{(24)}+b^{(14)}= \prod_{i,j=1;i<j}^3(\lambda_i- \lambda_j)^2\left(q_4^2
		- (3\lambda-2\lambda_4)\prod_{i=1}^{3}(\lambda_i -\lambda_4) \right).
	\end{aligned}
\end{equation}

More precisely, this is because, when viewing $q_4$ as a variable, the coefficient of the square term of $q_4$ in $b^{(34)}-b^{(24)}+b^{(14)}$ is
$$
\begin{aligned}
	T_{squ} &=\sum_{k=1}^3(-1) ^{k+1}\prod_{\substack{i,j=1;i<j;(i,j)\neq (k,4)}}^4(\lambda_{i}-\lambda_{j})(\lambda_{m}-\lambda_{n})=\prod_{i,j=1;i<j}^3(\lambda_{i}-\lambda_{j})^2,
\end{aligned}
$$
where $m<n$, $\{k,m,n\}=\{1,2,3\}$. Moreover, the coefficients of the mixed term of $q_4$ and the constant term of $q_4$ in $b^{(34)}-b^{(24)}+b^{(14)}$ are respectively

$$
\begin{aligned}
	T_{mix} =&-2\prod_{i,j=1;i<j}^{4} (\lambda_i - \lambda_j) \left(\sum_{i,j=1;i<j}^{3}(-1)^{i+j+1}(\lambda_{i}-\lambda_{j})\right.\\
	&\qquad\qquad\qquad\cdot(2\lambda_{i}+2\lambda_{j}+\lambda_{k})\Bigg)= 0\ \mbox{for}\ \{i,j,k\} =\{1,2,3\}
\end{aligned}
$$
%where $k=\{1,2,3\}\backslash\{i,j\}$.
and	
$$
\begin{aligned}
	T_{con} =&\prod_{i,j=1;i<j}^{4} (\lambda_i - \lambda_j) \left(\sum_{k,l=1;k<l}^{3}(-1)^{k+l+1}(\lambda_k - \lambda_l)(\lambda_m- \lambda_n)(2\lambda_k + 2\lambda_l+ \lambda_m)^2\right)\\
	= &-\prod_{i,j=1;i<j}^3(\lambda_i- \lambda_j)^2 \prod_{i=1}^3(\lambda_i- \lambda_4) (3\lambda-2\lambda_{4}),
\end{aligned}
$$
for $m<n$, $\{k,l,m,n\}=\{1,2,3,4\}$.

Similarly, we have
\begin{equation}\label{lnew2}
	\begin{aligned}
		&b^{(34)}+b^{(23)}-b^{(13)}= \prod_{i<j;i,j=1,2,4}(\lambda_i- \lambda_j)^2\left(q_3^2
-\prod_{i=1,2,4}(\lambda_i - \lambda_3) (3\lambda-2\lambda_3)\right),\\
		&b^{(23)}+b^{(12)}-b^{(24)}= \prod_{i<j;i,j=1,3,4}(\lambda_i- \lambda_j)^2\left(q_2^2+\prod_{i=1,3,4}(\lambda_i - \lambda_2) (3\lambda-2\lambda_2)\right),\\
		&b^{(12)}-b^{(13)}+b^{(23)}= \prod_{i<j;i,j=2,3,4}(\lambda_i- \lambda_j)^2\left(q_1^2
		-\prod_{i=2,3,4}(\lambda_i - \lambda_1) (3\lambda-2\lambda_1)\right).
	\end{aligned}
\end{equation}

By adding the four equations in $(\ref{lnew1})$ and $(\ref{lnew2})$, and observing that the sum of the terms independent of $q_k^2$ $(k=1,2,3,4)$ is zero, we have
\begin{equation*}\label{l}
\mathcal{A}=\sum_{k=1}^{4}\left(\prod_{i,j=1;i<j;i,j\neq k}^4(\lambda_i - \lambda_j)^2\right)\left(\lambda q_k^2-\prod_{i=1;i\neq k}^4(\lambda+\lambda_{i})q_k\right).
\end{equation*}
%$$
%\begin{aligned}
%	l= &6(\lambda_1 - \lambda_2)^2 (\lambda_1 - \lambda_3)^2 (\lambda_2 - \lambda_3)^2 \Bigl[ (\lambda_1 + \lambda_2 + \lambda_3 + \lambda_4) q_4^2 \\
%	&\qquad - (\lambda_1 + \lambda_2 + 2\lambda_3 + \lambda_4)(\lambda_1 + 2\lambda_2 + \lambda_3 + \lambda_4)(2\lambda_1 + \lambda_2 + \lambda_3 + \lambda_4) q_4 \Bigr] \\
%	&+ 6(\lambda_1 - \lambda_2)^2 (\lambda_1 - \lambda_4)^2 (\lambda_2 - \lambda_4)^2 \Bigl[ (\lambda_1 + \lambda_2 + \lambda_3 + \lambda_4) q_3^2 \\
%	&\qquad - (\lambda_1 + \lambda_2 + \lambda_3 + 2\lambda_4)(\lambda_1 + 2\lambda_2 + \lambda_3 + \lambda_4)(2\lambda_1 + \lambda_2 + \lambda_3 + \lambda_4) q_3 \Bigr] \\
%	&+ 6(\lambda_1 - \lambda_3)^2 (\lambda_1 - \lambda_4)^2 (\lambda_3 - \lambda_4)^2 \Bigl[ (\lambda_1 + \lambda_2 + \lambda_3 + \lambda_4) q_2^2 \\
%	&\qquad - (\lambda_1 + \lambda_2 + \lambda_3 + 2\lambda_4)(\lambda_1 + \lambda_2 + 2\lambda_3 + \lambda_4)(2\lambda_1 + \lambda_2 + \lambda_3 + \lambda_4) q_2 \Bigr] \\
%	&+6 (\lambda_2 - \lambda_3)^2 (\lambda_2 - \lambda_4)^2 (\lambda_3 - \lambda_4)^2 \Bigl[ (\lambda_1 + \lambda_2 + \lambda_3 + \lambda_4) q_1^2 \\
%	&\qquad - (\lambda_1 + \lambda_2 + \lambda_3 + 2\lambda_4)(\lambda_1 + \lambda_2 + 2\lambda_3 + \lambda_4)(\lambda_1 + 2\lambda_2 + \lambda_3 + \lambda_4) q_1 \Bigr]
%\end{aligned}
%$$

%Let
%$$
%S_k=\sum_{i\ne k}\lambda_i,\quad 
%Q_k=\sum_{ i,j=1; i<j; i,j\neq k}^4\lambda_i\lambda_j,\quad 
%P_k=\prod_{i\ne k}\lambda_i \ \mbox{for}\ k=1,2,3,4.
%$$
We choose $i,j,l$ with $\{i,j,l,k\}=\{1,2,3,4\}$. Then we have
$$\begin{aligned}
	\lambda q_k=&\lambda\left(\lambda_{k}^2 +2\sum_{i=1;i\neq k}^{4}\lambda_{i}^2+\sum_{i=1;i\neq k}^{4}\lambda_{i}\lambda_{k}+5\sum_{i,j=1;i<j;i,j\neq k}^{4}\lambda_{i}\lambda_{j}\right)\\
=&\left(\lambda_k+s_1^{ijl}\right)\left(\lambda_k^2+\lambda_ks_1^{ijl}+2\left(s_1^{ijl}\right)^2+s_2^{ijl}\right)\\
=&\lambda_k^3+2\lambda_k^2s_1^{ijl}+3\lambda_k\left(s_1^{ijl}\right)^2+2\left(s_1^{ijl}\right)^3+\lambda s_2^{ijl}
\end{aligned}
$$
and 
$$
\begin{aligned}
	\prod_{i=1;i\ne k}^4\left(\lambda+\lambda_i\right)
	=&\lambda^3+s_1^{ijl}\lambda^2+s_2^{ijl}\lambda+s_3^{ijl}\\
=&\left(\lambda_k+s_1^{ijl}\right)^3+s_1^{ijl}\left(\lambda_k+s_1^{ijl}\right)^2+s_2^{ijl}\left(\lambda_k+s_1^{ijl}\right)+s_3^{ijl}\\
	=&\lambda_k^3+4\lambda_k^2s_1^{ijl}+5\lambda_k(s_1^{ijl})^2+2(s_1^{ijl})^3+\lambda s_2^{ijl}+s_3^{ijl}.
\end{aligned}
$$
Subtracting the two expressions yields
$$
\lambda q_k-\prod_{i=1;i\ne k}^4(\lambda+\lambda_i)
=-2\lambda_k^2s_1^{ijl}-2\lambda_k\left(s_1^{ijl}\right)^2-s_3^{ijl}
=-2\lambda\lambda_k(\lambda-\lambda_k)-\prod_{i=1;i\ne k}^4\lambda_i.
$$

Consequently, we obtain
\begin{equation}\label{eql}
	\begin{aligned}
		\mathcal{A}=\sum_{k=1}^{4}\prod_{\substack{i,j=1;i<j;i,j\neq k}}^4\left(\lambda_{i}-\lambda_{j}\right)^2p_kq_k,
	\end{aligned}
\end{equation}
	in which
\begin{equation}\label{p}
p_k=-2\lambda\lambda_{k}\left(\lambda-\lambda_{k}\right)-\prod_{i=1;i\neq k}^{4}\lambda_{i}\ \mbox{for}\ k=1,2,3,4.
\end{equation}
	%\begin{equation*}
	%	\begin{aligned}
		%	p_4=& -2\lambda_4 (\lambda_1 + \lambda_2 + \lambda_3)(\lambda_1 + \lambda_2 + \lambda_3+\lambda_{4}) - \lambda_1 \lambda_2 \lambda_3,\\
		%	p_3=& -2\lambda_3 (\lambda_1 + \lambda_2 + \lambda_4)(\lambda_1 + \lambda_2 + \lambda_3+\lambda_{4}) - \lambda_1 \lambda_2 \lambda_4 ,\\
		%	p_2=&-2\lambda_2 (\lambda_1 + \lambda_3 + \lambda_4)(\lambda_1 + \lambda_2 + \lambda_3+\lambda_{4}) - \lambda_1 \lambda_3 \lambda_4 ,\\
		%	p_1=&-2\lambda_1 (\lambda_2 + \lambda_3 + \lambda_4)(\lambda_1 + \lambda_2 + \lambda_3+\lambda_{4}) - \lambda_2 \lambda_3 \lambda_4.
		%	\end{aligned}
	%\end{equation*}
	
%	Now we show that $\mathcal{A}=A(5)$. 
	%and we only illustrate for the term of $k=4$; the other three terms can be handled similarly. 
	Below, we use the polynomials $s_1^{ijk}, s_2^{ijk}$ and $s_3^{ijk}$
	defined in (\ref{s1}) to express $p_k$ and $q_k(i,j,k=1,2,3,4)$. We fix $l$ such that $\{i,j,k,l\}=\{1,2,3,4\}$. First, it is obvious that
\begin{equation}\label{pl1}
	\begin{aligned}
		p_l=&-2\lambda^2\lambda_l+2\lambda\lambda_l^2-s_3^{ijk}\\
		=&-2\lambda\lambda_l(\lambda-\lambda_l)-s_3^{ijk}.
	\end{aligned}
\end{equation}

		From $H=0$ we have
		$$\lambda_l=-s_1^{ijk}-\lambda_5=-s_1^{ijk}+\lambda,$$
		hence
		\begin{equation}\label{pl2}
			s_1^{ijk}=\lambda-\lambda_l. 
		\end{equation}
		
	Also, since
	$$
		\begin{aligned}
			\sigma_3=&s_3^{ijk}+\lambda_l s_2^{ijk}+\lambda_5\left(s_2^{ijk}+\lambda_l s_1^{ijk}\right)\\
			=&\left(s_3^{ijk}+\lambda_l s_2^{ijk}\right)-\lambda\left(s_2^{ijk}+\lambda_l s_1^{ijk}\right)\\
			=&s_3^{ijk}-s_1^{ijk}s_2^{ijk}-\lambda\lambda_l s_1^{ijk},
		\end{aligned}
$$
we have
\begin{equation}\label{pl3}
	-\lambda\lambda_l s_1^{ijk}=s_1^{ijk}s_2^{ijk}-s_3^{ijk}+\sigma_3. 	\end{equation}

	Combining (\ref{pl1}), (\ref{pl2}) and (\ref{pl3}), we obtain
	\begin{equation}\label{eqp4}
		p_l=2s_1^{ijk}s_2^{ijk}-3s_3^{ijk}+2\sigma_3,
	\end{equation}
	
Next, it follows from (\ref{qk}) that 
		\begin{equation}\label{eqq4new}
		q_l = 2\left(s^{ijk}_1\right)^2+s^{ijk}_2-\lambda_{l}\lambda_{5}.
	\end{equation}	
Since $H=0$, we have 
	\begin{equation}\label{eqC1}
		\begin{aligned}
			\sigma_2=&s_2^{ijk}+\left(\lambda_5 +\lambda_l\right)s_1^{ijk}+\lambda_l\lambda_5\\
			=&\lambda_{l}\lambda_{5}-\left(s^{ijk}_1\right)^2+s^{ijk}_2.
		\end{aligned}
	\end{equation}

	Hence, from $(\ref{eqq4new})$ and $(\ref{eqC1})$, we get
		\begin{equation}\label{eqq4}
	q_l = \left(s^{ijk}_1\right)^2+2s^{ijk}_2-\sigma_2.
\end{equation}	
	%\begin{equation}\label{eqC3}
	%	\begin{aligned}
		%	&\prod_{\substack{i<j\\i,j=1,2,3}}(\lambda_i - \lambda_j)^2 \\
		%	&=(d^{123}_1)^2(d^{123}_2)^2 - 4(d^{123}_2)^3 - 4(d^{123}_1)^3 d^{123}_3 
		%	- 27(d^{123}_3)^2 + 18d^{123}_1d^{123}_2d^{123}_3.
		%	\end{aligned}
	%\end{equation}
	
	Based on the combination of $(\ref{eql},\ref{eqp4},\ref{eqq4})$, we obtain \begin{equation}\label{Al}
		A(5)=\mathcal{A}.
	\end{equation}
Therefore, according to $(\ref{eqLl})$, we have
	%$$ L(5) > 0 \quad \iff  \quad l > 0\quad \iff C(5)>0.$$
	\begin{equation*}
	L(5)=\frac{A(5)}{	\prod_{i,j=1;i<j}^5(\lambda_{i}-\lambda_{j})^2}.
	\end{equation*}
	
Besides, it follows from $(\ref{L})$ that $L(r)$ $(r\in I_5)$ is obtained from $L(5)$ by replacing $\lambda_{r}$ with $\lambda_{5}$, and $A(r)$ $(r\in I_5)$, which is given in $(\ref{Ass})$, can also be obtained from $A(5)$ in the same way. Hence, we have
	\begin{equation}\label{eqAr}
L(r)=\frac{A(r)}{	\prod_{i,j=1;i<j}^5(\lambda_{i}-\lambda_{j})^2}
\end{equation}
for $r=1,2,\cdots,5$.

Consequently, combining $(\ref{letL})$ and $(\ref{eqAr})$, we obtain
	\begin{equation*}
		\begin{aligned}
		d\Phi =& -12\left(  f_3 +\frac{1}{25}\sum_{r=1}^5L(r)h_r^2\right)\cdot \mathrm{vol}\\
		=& -12 \left(3\,\sigma_3 +\frac{1}{25\prod_{i,j=1;i<j}^5(\lambda_{i}-\lambda_{j})^2}\sum_{r=1}^5A(r)h_r^2\right)\cdot \mathrm{vol}.
		\end{aligned}
	\end{equation*}
	\end{proof}

\section{Proof of Theorems}\label{Global thm}
\subsection{Proof of Theorem \ref{main thm}}\label{prof main thm}
\begin{proof}

When $\Omega$ is empty, it follows from \cite[3.2.1] {TWY}  that all principal curvatures of $M^5$ are constants, and the theorem holds. Therefore, in the following we only discuss the case that $\Omega\neq \emptyset$.

%\subsubsection{The differential of the $(n-1)$-form $\Phi$.}

%\subsubsection{The calculation of $dh\wedge\Phi$}	
%For convenience, in this subsection we compute the $5$-form  $dh\wedge\Phi$, which will be used in the proof of the theorem. 
%For $n=5$, $\Phi=\sum_{\sigma}S(\sigma)(\lambda_{i_{4}}+\lambda_{i_5})\theta_{i_1}\wedge\theta_{i_2}\wedge\theta_{i_{3}}\wedge\omega_{i_{4}i_{5}}.$

%\subsubsection{The condition $A(r)>0$.}\label{secC(r)}
%In this subsection, 
%we show that for any $r=1,2,\cdots,5$, the inequality $L(r) > 0$ is equivalent to $A(r)>0$, %whenever the condition $C(r)> 0$ holds for the elementary symmetric polynomials in the eigenvalues $\lambda_i$, for $r,i=1,2,\cdots,5$. 
%and we only prove the case $r=5$, the others are similar. 

	%\subsubsection{Proof that $h$ is constant}
	%When $n=5$,
	 Under our convention $\sigma_3\geq 0$, assumptions $(\ref{Ass})$ and $(\ref{key form})$, %the assumptions $A(r)>0$ for $r=1,2,\cdots,5$, 
	 we have 
	\begin{equation}\label{int1}
		\int_{Y}-d\Phi \geq 0.
	\end{equation}
	%\begin{equation}\label{u}
	%	dh\wedge\Phi=\sum_{i=1}^{5}u_ih_i^2.
	%\end{equation}
	
	As in \cite{dB}, for any smooth function $\eta:(a,b)\rightarrow \mathbb{R}$ with compact support, we apply Stokes' theorem to $$d\left(\left(\eta\circ h\right)\Phi\right)=\left(\eta\circ h\right)d\Phi+\left(\eta'\circ h\right)dh\wedge\Phi $$
	and obtain
	
	\begin{equation}\label{int2}
		\int_{Y}\left(\eta\circ h\right)d\Phi+\int_{Y}\left(\eta'\circ h\right)dh\wedge\Phi =0.
	\end{equation}
	
	Given a small $\epsilon>0$, we choose a smooth function $\eta_{\epsilon}:\mathbb{R}\rightarrow\mathbb{R}$ such that 
	\begin{itemize}
		\item[(1)] $0\leq\eta_{\epsilon}\leq 1$;
		\item[(2)] $\eta_{\epsilon}(t)=0$ for $a\leq t \leq a+\frac{\epsilon}{n}$ or $b-\frac{\epsilon}{n}\leq t \leq b$;
		\item[(3)] $\eta_{\epsilon}(t)=1$ for $a+\epsilon\leq t\leq b-\epsilon$;
		\item[(4)] $\eta_{\epsilon}'(t)\geq 0$ on $(-\infty,\frac{a+b}{2})$, and $\eta_{\epsilon}'(t)\leq 0$ on $(\frac{a+b}{2},+\infty)$.
	\end{itemize}
	
	Then it follows from $(\ref{u},\ref{int1},\ref{int2})$ that 
	\begin{equation}\label{int}
		\begin{aligned}
			0&\leq\int_{Y}-\left(\eta_{\epsilon}\circ h \right)d\Phi =\int_Y\left(\eta_{\epsilon}'\circ h\right )dh\wedge \Phi\\
			&=\int_Y\left(\eta_{\epsilon}'\circ h\right)\sum_{i=1}^nu_ih_i^2 \cdot \mathrm{vol}\leq\int_{Y}C\left|\eta_{\epsilon}'\circ h\right|\left|dh\right|^2\cdot \mathrm{vol},
		\end{aligned}
	\end{equation}
	where the last inequality follows from Lemma \ref{bound}.

	Moreover, for any smooth function $\gamma:\mathbb{R}\rightarrow\mathbb{R}$, applying Stokes' theorem to
	$$d^*\left(\left(\gamma\circ h\right)dh\right)=\left(\gamma'\circ h\right)\left|dh\right|^2\cdot \mathrm{vol}+\left(\gamma\circ h\right)\Delta h \cdot \mathrm{vol}$$
	yields
	\begin{equation}\label{int3}
		\int_{M^5}\left(\gamma'\circ h\right)\left|dh\right|^2\cdot \mathrm{vol}+\int_{M^5}\left(\gamma\circ h\right)\Delta h \cdot \mathrm{vol}=0.
	\end{equation}
	
	Next, we construct a new smooth function $\gamma_\epsilon$ from $\eta_{\epsilon}$, defined as follows.
	$$
	\gamma_\epsilon =
	\begin{cases} 
		\eta_\epsilon - 1 & \text{on} \left( -\infty, \frac{a+b}{2} \right], \\
		1 - \eta_\epsilon & \text{on} \left[ \frac{a+b}{2}, +\infty \right).
	\end{cases}
	$$
	It is obvious that $\gamma_\epsilon'=\left|\eta_\epsilon'\right|$. Then it follows from $(\ref{int3})$ that 
	$$\int_Y\left|\eta_\epsilon'\circ h\right| \left|dh\right|^2\cdot \mathrm{vol}=-\int_{M^5}\left(\gamma_\epsilon\circ h\right)\Delta h \cdot \mathrm{vol}\leq \int_{M^5}\left|\gamma_\epsilon\circ h\right|\left|\Delta h\right|\cdot \mathrm{vol}.$$

	From the construction of $\gamma_\epsilon$, we have $|\gamma_\epsilon|\leq 1$ and $\gamma_\epsilon\circ h=0$ on $Y_\epsilon$. Thus by Lemma \ref{lemdB}, we obtain
	\begin{equation*}
		\begin{aligned}
			\lim_{\epsilon\to0}\int_{M^5}|\gamma_\epsilon\circ h||\Delta h| \cdot \mathrm{vol}=\lim_{\epsilon\to0}\int_{M^5-Y_\epsilon}|\gamma_\epsilon\circ h||\Delta h| \cdot \mathrm{vol}\leq\lim_{\epsilon\to0}\int_{M^5-Y_\epsilon}|\Delta h| \cdot \mathrm{vol}=0.
		\end{aligned}
	\end{equation*}
	
	Furthermore,
	$$\lim_{\epsilon\to0}\int_{Y}\left|\eta_\epsilon'\circ h\right|\left|dh\right|^2\cdot \mathrm{vol}=0.$$
	Together with $(\ref{int})$, we obtain
	$$\lim_{\epsilon\to0}\int_{Y}-\left(\eta_\epsilon\circ h\right)d\Phi=0.$$
	
	Finally, according to the convention $\sigma_3\geq 0$, $(\ref{key form})$ and assumption $(\ref{Ass})$, we have
	$$0\leq \int_{Y_{\epsilon'}}\frac{12}{25\prod_{i,j=1;i<j}^5(\lambda_{i}-\lambda_{j})^2}\sum_{r=1}^5A(r)h_r^2\cdot \mathrm{vol}\leq \int_{Y}-\left(\eta_\epsilon\circ h\right)d\Phi$$
	for all $0<\epsilon\leq\epsilon'<\frac{b-a}{2}$. %Moreover,under the condition that %$C(r)>0,(r=1,\cdots,5)$, we have $L(r)>0$ on $Y$ for $r=1,2,3,4,5$, 
	Thus we must have $h_r=0$ on $Y$ for any $r=1,2,\cdots, 5$. Thus, $h$ is locally constant on each connected component of $Y$. Since $Y$ is dense in the connected manifold $M^5$ and $h$ is continuous, it follows that 
$h$ is constant on $M^5$.

		We remark that if $(\lambda_1,\lambda_{2},\lambda_{3},\lambda_{4},\lambda_{5})=(-6,-5,1,3,7)$, we have $$A(i) > 0\,(i\in I_5)\ \mbox{and}\ A(5) < 0,$$ so not all $\lambda_i$ $(i=1,2,\cdots,5)$ satisfy $(\ref{Ass})$. %Then the proof of Theorem \ref{main thm} is completed.

	\end{proof}

	\subsection{Proof of Corollary \ref{cor}}\label{prof cor}
	\begin{proof}
		In this proof,  for the first three configurations, we use the form of $A(5)=\mathcal{A}$ given in (\ref{eql}); the other $A(r) (r\in I_5)$ are analogous, and we note that for each $i=1,2,\cdots,5$, this expression for $A(i)$ does not contain $\lambda_i$.
		We only need to show that for the four principal curvature configurations $(1)$, $(2)$, $(3)$ and $(4)$, the inequality $A(r)>0$ holds for every $r=1,2,\cdots,5$.
		%We only need to prove that Assumption $(3)$ holds when the principal curvatures belong to $(1'')$. Below we verify separately that for each class in $(1'')$, both $C(r) > 0$ for all $r=1,2,\cdots,5$ and $\sigma_3 \geq 0$ hold. 
		
		%In the case where
		
		Case (1): $\lambda_{1}<\lambda_{2}<\lambda_{3}=\lambda_{4}<0<\lambda_{5}$. %we always have Assumption $(3)$ holds, and at the same time $\sigma_3\geq 0$. This is because at this time: 
		In this case, substituting $\lambda_{4}$ for $\lambda_{3}$ in $A(5)$, we get 
		\begin{equation}\label{case1A5}
			\begin{aligned}
				A(5)=-2\lambda_{4}\prod_{i,j=1,2,4;i<j}{\left(\lambda_{i}-\lambda_{j}\right)}^2\mathcal{P}_1^2
			\end{aligned}
		\end{equation}
		where $$\mathcal{P}_1=2{\lambda_{1}}^2+5\lambda_{1}\lambda_{2}+6\lambda_{1}\lambda_{4}+2{\lambda_{2}}^2+6\lambda_{2}\lambda_{4}+4{\lambda_{4}}^2>0.$$
		
		Besides, $A(1)$ and $A(2)$ are the results of replacing $\lambda_{1}$ or $\lambda_{2}$ in $A(5)$ with $\lambda_{5}$, respectively. Thus it is clear that $A(1), A(2)$ and $A(5) >0$.
		
		Next, replacing $\lambda_{5}$ with $-\lambda_{1}-\lambda_{2}-2\lambda_{4}$, we obtain
		\begin{equation}\label{case1A3}
			\begin{aligned}
				A(3)=-{\left(\lambda_{1}-\lambda_{2}\right)}^2\mathcal{P}_1^2\cdot \mathcal{P}_2\left(\lambda_1,\lambda_{2},\lambda_{4}\right),
			\end{aligned}
		\end{equation}
		where 
		$$
		\begin{aligned}
			\mathcal{P}_2\left(\lambda_1,\lambda_{2},\lambda_{4}\right)=&{\lambda_{1}}^4\lambda_{2}+{\lambda_{1}}^4\lambda_{4}+2{\lambda_{1}}^3{\lambda_{2}}^2+6{\lambda_{1}}^3\lambda_{2}\lambda_{4}+4{\lambda_{1}}^3{\lambda_{4}}^2+2{\lambda_{1}}^2{\lambda_{2}}^3\\
			&+9{\lambda_{1}}^2{\lambda_{2}}^2\lambda_{4}-5{\lambda_{1}}^2\lambda_{2}{\lambda_{4}}^2-4{\lambda_{1}}^2{\lambda_{4}}^3+\lambda_{1}{\lambda_{2}}^4+6\lambda_{1}{\lambda_{2}}^3\lambda_{4}\\
			&-5\lambda_{1}{\lambda_{2}}^2{\lambda_{4}}^2-34\lambda_{1}\lambda_{2}{\lambda_{4}}^3-16\lambda_{1}{\lambda_{4}}^4+{\lambda_{2}}^4\lambda_{4}+4{\lambda_{2}}^3{\lambda_{4}}^2\\
			&-4{\lambda_{2}}^2{\lambda_{4}}^3-16\lambda_{2}{\lambda_{4}}^4+47{\lambda_{4}}^5,
		\end{aligned}$$
		and $A(4)$ is obtained by replacing $\lambda_{4}$ in $A(3)$ with $\lambda_{3}$. 
		
		Now we claim that $\mathcal{P}_2(\lambda_1,\lambda_{2},\lambda_{4})<0$, and thus $A(3)>0$.
		Note that $\mathcal{P}_2$ is a polynomial of degree $5$, so it suffices to prove that $f(\lambda') = -\mathcal{P}_2(\lambda_1,\lambda_{2},\lambda_{4}) > 0$ for  $\lambda'=(\lambda_1',\lambda_{2}',\lambda_{4}')$ with $\lambda_1' > \lambda_2' > \lambda_4' > 0$, where we set $$\lambda_{1}'=-\lambda_{1},\lambda_{2}'=-\lambda_{2}\ \mbox{and}\ \lambda_{4}'=-\lambda_{4}.$$ 
		
		Let $$x_1=\frac{\lambda_{1}'}{\lambda_{4}'}\ \mbox{and}\ x_2=\frac{\lambda_{2}'}{\lambda_{4}'}.$$ Then $x_1 > x_2 > 1$, and $f(\lambda') =(\lambda_4')^5 g\left(x_1,x_2\right)$, where 
		\begin{equation}\label{case1g}
		\begin{aligned}
			g\left(x_1,x_2\right)=&	{x_{1}}^4x_{2}+{x_{1}}^4+2{x_{1}}^3{x_{2}}^2+6{x_{1}}^3x_{2}+4{x_{1}}^3+2{x_{1}}^2{x_{2}}^3\\
			&+9{x_{1}}^2{x_{2}}^2-5{x_{1}}^2x_{2}-4{x_{1}}^2+x_{1}{x_{2}}^4+6x_{1}{x_{2}}^3-5x_{1}{x_{2}}^2\\
			&-34x_{1}x_{2}-16x_{1}+{x_{2}}^4+4{x_{2}}^3-4{x_{2}}^2-16x_{2}+47.
		\end{aligned}
	\end{equation}
		
		Therefore, it only remains to show that $g(x_1,x_2) > 0$ for all $x_1 > x_2 > 1$.
		By direct calculation, 
		$$
		\begin{aligned}
			&\frac{\partial g}{\partial x_1}=4{x_{1}}^3x_{2}+4{x_{1}}^3+6{x_{1}}^2{x_{2}}^2+18{x_{1}}^2x_{2}+12{x_{1}}^2\\
			&\qquad\quad+4x_{1}{x_{2}}^3+18x_{1}{x_{2}}^2-10x_{1}x_{2}-8x_{1}+{x_{2}}^4\\
			&\qquad\quad+6{x_{2}}^3-5{x_{2}}^2-34x_{2}-16,\\
			&\frac{\partial^2g}{\partial x_1^2}=12{x_{1}}^2x_{2}+12{x_{1}}^2+12x_{1}{x_{2}}^2+36x_{1}x_{2}+24x_{1}\\
			&\qquad\quad+4{x_{2}}^3+18{x_{2}}^2-10x_{2}-8.
		\end{aligned}
		$$
		
		Since $x_1>x_2>1$, it is obvious that $	\frac{\partial^2g}{\partial x_1^2}>0$. Thus $\frac{\partial g}{\partial x_1}$ is strictly increasing in $x_1$, and hence
		$$\begin{aligned}
			\frac{\partial g}{\partial x_1}\left(x_1,x_2\right)&>\frac{\partial g}{\partial x_1}\left(x_2,x_2\right)\\
			&=\left(x_2-1\right)\left(15{x_{2}}^3+61{x_{2}}^2+58x_{2}+16\right)>0.
		\end{aligned}
		$$
		
		Thus, it follows that
		$$\begin{aligned}
			g(x_1,x_2)&>g(x_2,x_2)\\
			&={\left(x_{2}-1\right)}^2\left(6{x_{2}}^3+35{x_{2}}^2+62x_{2}+47\right)>0.
		\end{aligned}
		$$
		
		Therefore, the claim holds. By a completely analogous argument, we can prove that $A(4)>0$. Finally, we point out that in this case, $\sigma_3\geq 0$ is obvious.
		
		Case (2): $\lambda_{1}=\lambda_{2}<\lambda_{3}<\lambda_{4}<\lambda_{5}$.  Similar to the Case (1), substituting $\lambda_{1}$ for $\lambda_{2}$, we can directly obtain $A(3),A(4), A(5)>0$ and for $A(1)$ $(\mbox{or}\ A(2))$, replacing $\lambda_{2}$ $(\mbox{or}\ \lambda_{1})$ with $-\frac{1}{2}(\lambda_{3}+\lambda_{4}+\lambda_{5})$, we get 
		$$32A(1)=32A(2)=\prod_{i,j=3,4,5;i<j}{({\lambda}_{i}-{\lambda}_{j})}^{2}\mathcal{P}_3(\lambda_3,\lambda_{4},\lambda_{5}),$$
		where 
		$$
		\begin{aligned}
			\mathcal{P}_3(\lambda_3,\lambda_{4},\lambda_{5})=&47{\lambda}_{3}^{5}+267{\lambda}_{3}^{4}{\lambda}_{4}+267{\lambda}_{3}^{4}{\lambda}_{5}+582{\lambda}_{3}^{3}{\lambda}_{4}^{2}+1060{\lambda}_{3}^{3}{\lambda}_{4}{\lambda}_{5}+582{\lambda}_{3}^{3}{\lambda}_{5}^{2}\\
			&+582{\lambda}_{3}^{2}{\lambda}_{4}^{3}+1570{\lambda}_{3}^{2}{\lambda}_{4}^{2}{\lambda}_{5}+1570{\lambda}_{3}^{2}{\lambda}_{4}{\lambda}_{5}^{2}+582{\lambda}_{3}^{2}{\lambda}_{5}^{3}+267{\lambda}_{3}{\lambda}_{4}^{4}\\
			&+1060{\lambda}_{3}{\lambda}_{4}^{3}{\lambda}_{5}+1570{\lambda}_{3}{\lambda}_{4}^{2}{\lambda}_{5}^{2}+1060{\lambda}_{3}{\lambda}_{4}{\lambda}_{5}^{3}+267{\lambda}_{3}{\lambda}_{5}^{4}+47{\lambda}_{4}^{5}\\
			&+267{\lambda}_{4}^{4}{\lambda}_{5}+582{\lambda}_{4}^{3}{\lambda}_{5}^{2}+582{\lambda}_{4}^{2}{\lambda}_{5}^{3}+267{\lambda}_{4}{\lambda}_{5}^{4}+47{\lambda}_{5}^{5}.
		\end{aligned}$$
		
		Thus it is clear that $A(1)=A(2)>0$ when $0<\lambda_{3}<\lambda_{4}<\lambda_{5}$. Now we claim that  $\mathcal{P}_3(\lambda_3,\lambda_{4},\lambda_{5})>0$ when $\lambda_{3}<0<\lambda_{4}<\lambda_{5}$ or $\lambda_{3}<\lambda_{4}<0<\lambda_{5}$, and then  $A(1)=A(2)>0$. Firstly, we notice that
		$$0>\lambda_{3}>-\frac{1}{3}(\lambda_{4}+\lambda_{5}), \lambda_{4}>-\frac{1}{4}\lambda_{5}\ \mbox{and}\ \lambda_{4}+\lambda_{5}>0.$$ 
		
		The third partial derivative of $\mathcal{P}_3(\lambda_3,\lambda_{4},\lambda_{5})$ with respect to $\lambda_{3}$ is
		$$
		\begin{aligned}
			\frac{\partial^3 \mathcal{P}_3}{\partial \lambda_{3}^3}=&2820{\lambda_{3}}^2+6408\lambda_{3}\left(\lambda_{4}+\lambda_{5}\right)+3492{\lambda_{4}}^2+6360\lambda_{4}\lambda_{5}+3492{\lambda_{5}}^2\\
			>&2820{\lambda_{3}}^2-2160\left(\lambda_{4}+\lambda_{5}\right)^2+3492{\lambda_{4}}^2+6360\lambda_{4}\lambda_{5}+3492{\lambda_{5}}^2\\
			=&2820{\lambda_{3}}^2+1020\left(\lambda_{4}+\lambda_{5}\right)^2+312{\lambda_{4}}^2+312{\lambda_{5}}^2
			>0,
		\end{aligned}$$
		thus  
		$$
		\begin{aligned}
			\frac{\partial^2 \mathcal{P}_3}{\partial \lambda_{3}^2}&>\frac{\partial^2 \mathcal{P}_3}{\partial \lambda_{3}^2}\left(-\frac{1}{3}(\lambda_{4}+\lambda_{5}),\lambda_{4},\lambda_{5}\right)\\
			&=\frac{8672\lambda_{4}^3}{27}+\frac{7376{\lambda_{4}}^2\lambda_{5}}{9}+\frac{7376\lambda_{4}{\lambda_{5}}^2}{9}+\frac{8672{\lambda_{5}}^3}{27}\\
			&=\left(\lambda_{4}+\lambda_{5}\right)\left(\frac{8672{\lambda_{4}}^2}{27}+\frac{13456\lambda_{4}\lambda_{5}}{27}+\frac{8672{\lambda_{5}}^2}{27}\right)>0.
		\end{aligned}$$
		
		From this we obtain 
		$$
		\begin{aligned}
			\frac{\partial \mathcal{P}_3}{\partial \lambda_{3}}&>	\frac{\partial \mathcal{P}_3}{\partial \lambda_{3}}\left(-\frac{1}{3}(\lambda_{4}+\lambda_{5}),\lambda_{4},\lambda_{5}\right)\\
			&=\frac{2944{\lambda_{4}}^4}{81}+\frac{17824{\lambda_{4}}^3\lambda_{5}}{81}+\frac{9488{\lambda_{4}}^2{\lambda_{5}}^2}{27}+\frac{17824\lambda_{4}{\lambda_{5}}^3}{81}+\frac{2944{\lambda_{5}}^4}{81}\\
			&=\frac{16}{81}\lambda_{5}^4\left(\frac{\lambda_{4}}{\lambda_{5}}+1\right)\left(4\frac{\lambda_{4}}{\lambda_{5}}+1\right)\left(46\left(\frac{\lambda_{4}}{\lambda_{5}}\right)^2+ 83\frac{\lambda_{4}}{\lambda_{5}} + 46\right)>0,
		\end{aligned}$$
		therefore
		$$
		\begin{aligned}
			\mathcal{P}_3\left(\lambda_3,\lambda_{4},\lambda_{5}\right)&>\mathcal{P}_3\left(-\frac{1}{3}(\lambda_{4}+\lambda_{5}),\lambda_{4},\lambda_{5}\right)\\
			&=\frac{64\left(\lambda_{4}+\lambda_{5}\right){\left(4{\lambda_{4}}^2+17\lambda_{4}\lambda_{5}+4{\lambda_{5}}^2\right)}^2}{243}>0.
		\end{aligned}$$
		The claim holds. 
		
		Case (3): $\lambda_{1}<\lambda_{2}=\lambda_{3}<\lambda_{4}<0<\lambda_{5}$.
		Likewise, as in Case (1), substituting $\lambda_{2}$ for $\lambda_{3}$, we can directly obtain $A(1),A(4),A(5)>0$, and in $A(2)$ $(\mbox{or}\ A(3))$, replacing $\lambda_{5}$ by $-\lambda_{1}-2\lambda_{3}-\lambda_{4}$ $(\mbox{or}\ \lambda_{1}-2\lambda_{2}-\lambda_{4})$, we obtain
		$$
		\begin{aligned}
			A(2)=- (\lambda_1 - \lambda_4)^2 \mathcal{P}_0^2\cdot \mathcal{P}_4\left(\lambda_1,\lambda_{3},\lambda_{4}\right),
		\end{aligned}
		$$	
		where 
		$$\begin{aligned}
			&\mathcal{P}_0=2\lambda_1^2 + 6\lambda_1\lambda_3 + 5\lambda_1\lambda_4 + 4\lambda_3^2 + 6\lambda_3\lambda_4 + 2\lambda_4^2,\\
			&\mathcal{P}_4\left(\lambda_1,\lambda_{3},\lambda_{4}\right)=\lambda_1^4\lambda_3 + \lambda_1^4\lambda_4 + 4\lambda_1^3\lambda_3^2 + 6\lambda_1^3\lambda_3\lambda_4+ 2\lambda_1^3\lambda_4^2- 4\lambda_1^2\lambda_3^3  \\
			&\qquad\qquad\qquad\qquad- 5\lambda_1^2\lambda_3^2\lambda_4 + 9\lambda_1^2\lambda_3\lambda_4^2+ 2\lambda_1^2\lambda_4^3- 16\lambda_1\lambda_3^4 - 34\lambda_1\lambda_3^3\lambda_4 \\
			&\qquad\qquad\qquad\qquad- 5\lambda_1\lambda_3^2\lambda_4^2 + 6\lambda_1\lambda_3\lambda_4^3 + \lambda_1\lambda_4^4+ 47\lambda_3^5 - 16\lambda_3^4\lambda_4\\
			&\qquad\qquad\qquad\qquad - 4\lambda_3^3\lambda_4^2 + 4\lambda_3^2\lambda_4^3 + \lambda_3\lambda_4^4,
		\end{aligned}
		$$
		and $A(3)$ is obtained by replacing $\lambda_{3}$ in $A(2)$ with $\lambda_{2}$.
		
		Since $\mathcal{P}_4\left(\lambda_1,\lambda_{3},\lambda_{4}\right)$ is homogeneous,  we can set  $d=-\lambda_{1},\lambda_3=-1,c=-\lambda_{4}$, so that $d>1>c>0$ and $\mathcal{P}_4\left(\lambda_1,\lambda_{3},\lambda_{4}\right)=-f(c,d)$, where
		$$
		\begin{aligned}
			f(c,d)=&c^4d+c^4+2c^3d^2+6c^3d+4c^3+2c^2d^3+9c^2d^2\\
			&-5c^2d-4c^2+cd^4+6cd^3-5cd^2-34cd\\
			&-16c+d^4+4d^3-4d^2-16d+47.
		\end{aligned}$$
		
		Let $a=d-1>0,\, b=1-c\in (0,1)$. Then we get
		$$\begin{aligned}
			g(a,b)=f(c,d)=&\left(2-b\right)a^4+\left(2b^2-14b+20\right)a^3+\left(-2b^3+21b^2-55b+50\right)a^2\\
			&+\left(b^4-14b^3+55b^2-50b\right)a+2b^4-20b^3+50b^2.
		\end{aligned}
		$$
		
		Since $$2-b>0,\,2b^2-14b+20>0,\, -2b^3+21b^2-55b+50>0$$
		and$$
		\begin{aligned}
			\Delta=&\left(b^4-14b^3+55b^2-50b\right)^2-4\left(-2b^3+21b^2-55b+50\right)\\
			&\qquad\qquad\qquad\qquad\qquad\qquad\qquad\qquad\cdot\left(2b^4-20b^3+50b^2\right)\\
			=&-b^2\left(b - 5\right)^2\left(- b^4 + 2b^3 + 67b^2 - 260b + 300\right)<0,
		\end{aligned}$$
		we have $g(a,b)>0$, and therefore $A(2)>0$. Analogously, $A(3)>0$.

		Case (4): $\lambda_{1}<\lambda_{2}<\lambda_{3}<\lambda_{4}<0<\lambda_{5}$. In this case, it is clear that $A(5) > 0$ by Lemma \ref{pvec}. For the other $A(r)$ $(r=1,2,3,4)$, we only prove $A(4) > 0 $, and the rest follow similarly.
		
		%From Lemma \ref{pvec}, we know that $A(4)$ is the result of replacing $\lambda_4$ in $ A(5)$ with $\lambda_5$, and then
		 Replacing $\lambda_5$ with $-\lambda$, Lemma \ref{pvec} yields
		
		\begin{equation}\label{A4}
			\begin{aligned}
				A(4) =& \lambda\prod_{i,j=1;i<j}^3\left(\lambda_{i}-\lambda_{j}\right)^{2}q_4^2-\sum_{k=1}^{3}\lambda_{k}\rho_k^2 \left(\lambda_{m}-\lambda_{l}\right)^{2}\left(\lambda_m+\lambda\right)^{2}\left(\lambda_{l}+\lambda\right)^{2}
			\end{aligned}
		\end{equation}
		for $\{m,l,k \}= \{1,2,3\}$,
		where 
		$$\rho_k=\sum_{i=1;i\neq k}^{3}\left(\lambda_{i}^2+\lambda_{i}\lambda_{4}\right)+\sum_{i,j=1,i<j}^{3}\lambda_{i}\lambda_{j}-2\left(\lambda_k^2+\lambda_{4}^2\right)-3\lambda_{k}\lambda_{4}$$ %\mbox{for}\ k=1,2,3.$$
		for $k=1,2,3$.
		%	$$
		%	\begin{aligned}
			%T_4=q_4
			%\lambda_{1}^{2}+5\lambda_{1}\lambda_{2}+5\lambda_{1}\lambda_{3}+\lambda_{1}\lambda_{4}+2\lambda_{2}^{2}+5\lambda_{2}\lambda_{3}+\lambda_{2}\lambda_{4}+2\lambda_{3}^{2}+\lambda_{3}\lambda_{4}+\lambda_{4}^{2} \\
			%	&T_2=\lambda_{1}^{2}+\lambda_{1}\lambda_{2}+\lambda_{1}\lambda_{3}+\lambda_{1}\lambda_{4}-2\lambda_{2}^{2}+\lambda_{2}\lambda_{3}-3\lambda_{2}\lambda_{4}+\lambda_{3}^{2}+\lambda_{3}\lambda_{4}-2\lambda_{4}^{2}\\
			%	&T_3=\lambda_{1}^{2}+\lambda_{1}\lambda_{2}+\lambda_{1}\lambda_{3}+\lambda_{1}\lambda_{4}+\lambda_{2}^{2}+\lambda_{2}\lambda_{3}+\lambda_{2}\lambda_{4}-2\lambda_{3}^{2}-3\lambda_{3}\lambda_{4}-2\lambda_{4}^{2} \\
			%	&T_1=-2\lambda_{1}^{2}+\lambda_{1}\lambda_{2}+\lambda_{1}\lambda_{3}-3\lambda_{1}\lambda_{4}+\lambda_{2}^{2}+\lambda_{2}\lambda_{3}+\lambda_{2}\lambda_{4}+\lambda_{3}^{2}+\lambda_{3}\lambda_{4}-2\lambda_{4}^{2}
			%	\end{aligned}$$.

		Write $$
		\begin{aligned}
			\mathcal{P}_5=&-\lambda_{2}(\lambda+\lambda_{3})^{2}(\lambda+\lambda_{1})^{2}\rho_2^2+ (\lambda_{1}+\lambda_{2})(\lambda_{1}-\lambda_{2})^{2}(\lambda_{2}-\lambda_{3})^{2}q_4^2,\\
			\mathcal{P}_6=&- \lambda_{3}(\lambda+\lambda_{2})^{2}(\lambda+\lambda_{1})^{2}\rho_3^2+4\lambda_{3}(\lambda_{1}-\lambda_{3})^{2}(\lambda_{2}-\lambda_{3})^{2}q_4^2.
		\end{aligned}$$
We now prove separately that $\mathcal{P}_5> 0$ and $\mathcal{P}_6 > 0$, and then
		$$A(4)> \left(\lambda_{1}-\lambda_{3}\right)^2\mathcal{P}_5+\left(\lambda_{1}-\lambda_{2}\right)^2\mathcal{P}_6>0.$$			% from which it naturally follows that $C(4) > 0$.  
		
		Let $$d=-\lambda_{1},\quad c=-\lambda_{2}, \quad b=-\lambda_{3},\quad a=-\lambda_{4}.$$ Then $d>c>b>a>0$ and we have 
		$$
		\begin{aligned}
			\mathcal{P}_5=&c{\left(a+b+c+2d\right)}^2{\left(a+2b+c+d\right)}^2\\
			&\qquad\cdot{\left(-2a^2+ab-3ac+ad+b^2+bc+bd-2c^2+cd+d^2\right)}^2\\
			&-\left(c+d\right){\left(b-c\right)}^2{\left(c-d\right)}^2\\
			&\qquad\cdot{\left(a^2+ab+ac+ad+2b^2+5bc+5bd+2c^2+5cd+2d^2\right)}^2,
		\end{aligned}$$
		$$\begin{aligned}
			\mathcal{P}_6=&b\left (a+b+c+2d\right)^2 \left(a+b+2c+d\right)^2\left(\mathcal{P}^{(1)}_{6}\right)^2
		 - 4b \left(b-c\right)^2 \left(b-d\right)^2\left(\mathcal{P}^{(2)}_{6}\right)^2.\\
		\end{aligned}
		$$
	where
		$$\mathcal{P}^{(1)}_{6}=-2a^2-3ab+ac+ad-2b^2+bc+bd+c^2+cd+d^2>0,$$
		$$\mathcal{P}^{(2)}_{6}=a^2+ab+ac+ad+2b^2+5bc+5bd+2c^2+5cd+2d^2>0.
		$$
		Therefore, for $\mathcal{P}_6>0$, it suffices to show that %one of the factors in its factorization satisfies 
		$$\begin{aligned}
			\mathcal{P}_{6}^{fac} =&\left(a+b+c+2d\right) \left(a+b+2c+d\right)\mathcal{P}^{(1)}_{6}\\
			&- 2\left(b-c\right) \left(b-d\right)\mathcal{P}^{(2)}_{6}>0.
		\end{aligned}$$
		
		Taking the second derivative of both sides of $\mathcal{P}_{6}^{fac} $ with respect to $d$ yields
		$$\begin{aligned}
			\frac{\partial^2 \mathcal{P}_6^{fac}   }{\partial d^2}=&12 b^2 + 34 b c + 54 b d + 8 a b - 2 c^2 \\
			&+ 18 c d + 22 a c + 24 d^2 + 30 a d
			>0.
		\end{aligned}$$
		Thus 
		$$\begin{aligned}
			\frac{\partial \mathcal{P}_6^{fac}  }{\partial d}>\left.\frac{\partial \mathcal{P}^{fac}   }{\partial d}\right|_{d=c}=&-5 a^3 - 10 a^2 b - 5 a^2 c- 12 a b^2 + 9 a b c\\
			&+ 48 a c^2 - 11 b^3+ 15 b^2 c + 78 b c^2 + 18 c^3 
			>0.
		\end{aligned}
		$$
	Consequently, we have
		$$
		\begin{aligned}
		\mathcal{P}_{6}^{fac}  >\left. \mathcal{P}_{6}^{fac}  \right|_{d=c}=&(-2a^4-7a^3b+9ac^3)+(-20a^2bc+9abc^2+11bc^3)\\
			&+(-6b^4+bc^3+4b^2c^2+bc^3)+(-10a^3c+10ac^3)\\
			&+(-12a^2b^2+ac^3+11b^2c^2)+(-9ab^3+9bc^3)\\
			&+(-5a^2c^2+4bc^3+ac^3)+(-24ab^2c+13bc^3+11ac^3)\\
			&+(-22b^3c+9c^4+13bc^3)
			>0.
		\end{aligned}
		$$
		
		For $\mathcal{P}_5>0$, by direct calculation, we obtain
		$$\frac{\partial^7 \mathcal{P}_5}{\partial d^7}=20160\left(5ac+9bc+8cd-b^2+4c^2\right)>0,$$
		and thus 
		$$\begin{aligned}
			\frac{\partial^6 \mathcal{P}_5}{\partial d^6}>\left.\frac{\partial^6 \mathcal{P}_5}{\partial d^6}\right|_{d=c}=&2880\left(
			6.25a^2 c - a b^2 + 32.5a b c\right.  \\
			&\left.+48.5 a c^2 - 5b^3 + 20.25b^2 c + 96.5 b c^2 + 58.25 c^3\right)>0.
		\end{aligned}
		$$
	By analogy, 
		$$\begin{aligned}
			\frac{\partial^5 \mathcal{P}_5}{\partial d^5}>&\left.\frac{\partial^5 \mathcal{P}_5}{\partial d^5}\right|_{d=c}>%70200c^4-2400a^3c-600a^2b^2-1680ab^3-3960b^4>0,
			600\left(117c^4-4a^3c-a^2b^2-2.8ab^3-6.6b^4\right)>0	\end{aligned}$$
				$$	\Downarrow$$
			$$\begin{aligned}
			\frac{\partial^4 \mathcal{P}_5}{\partial d^4}>\left.\frac{\partial^4 \mathcal{P}_5}{\partial d^4}\right|_{d=c}>&-1392a^4c-48a^3b^2-2784a^3bc-7200a^3c^2-288a^2b^3\\
			&-1752a^2c^3-336ab^4-480b^5+19224c^5>0,	\end{aligned}$$
				$$	\Downarrow$$
			$$\begin{aligned}
			\frac{\partial^3 \mathcal{P}_5}{\partial d^3}>\left.\frac{\partial^3 \mathcal{P}_5}{\partial d^3}\right|_{d=c}>&-120 a^5 c - 6 a^4 b^2 - 1236 a^4 b c - 2058 a^4 c^2 - 12 a^3 b^3- 1452 a^3 b^2 c \\
			& - 7128 a^3 bc^2  - 7308 a^3 c^3 - 30 a^2 b^4 - 3360 a^2 b c^3- 7008 a^2 c^4   \\
			&- 24 a b^5 - 24 b^6 + 38142 b^2 c^4 + 24012 b c^5 + 3402 c^6>0,	\end{aligned}$$
				$$	\Downarrow$$
			$$\begin{aligned}
			\frac{\partial^2 \mathcal{P}_5}{\partial d^2}>\left.\frac{\partial^2 \mathcal{P}_5}{\partial d^2}\right|_{d=c}>&2c\Big[- 22 a^5 b - 289 a^4 b^2 - 792 a^4 b c - 444 a^4 c^2- 242 a^3 b^3 - 1718 a^3 b^2 c \\ &+\left(- 1752 a^3 c^3 - 3392 a^2 b c^3+ 5444 b^3 c^3\right) +\left(- 468 a c^5+ 2196 b c^5 \right)\\
			&+\left(- 2023 a^2 c^4 + 5342 b^2 c^4  \right)- 3238 a^3 b c^2   - 1065 a^2 b^2 c^2+ 2777 b^4 c^2 \\
			& + 162 c^6\Big]>0,	\end{aligned}$$
				$$	\Downarrow$$
			$$
			\begin{aligned}
			\frac{\partial \mathcal{P}_5}{\partial d}>\left.\frac{\partial \mathcal{P}_5}{\partial d}\right|_{d=c}=&c (a-b)^2 (2a+b+2c)^2 (2a+4b+4c) (a+b+3c)^2\\
			& - 2c (a-b) (a+2b+2c)^2 (2a+b+2c) (a+b+3c)^3\\
			& + c (a-b)^2 (a+2b+2c)^2 (2a+b+2c)^2 (4a+4b+12c)>0,
			\end{aligned}
		$$
and hence we have
		$$
		\mathcal{P}_5>\left.\mathcal{P}_5\right|_{d=c}=
		c{\left(a-b\right)}^2{\left(a+2b+2c\right)}^2{\left(2a+b+2c\right)}^2{\left(a+b+3c\right)}^2>0.
		$$
The fully explicit forms for the derivatives of $\mathcal{P}_5$ are provided in Appendix \ref{B}.
	\end{proof}

\subsection{Proof of Theorem \ref{two}}
\begin{proof}
	We first point out that, from Newton's formula, it follows that %for a minimal hypersurface,
	 the constancy of $H_3$ is equivalent to that of $f_3$, and when the scalar curvature $R$ is constant, $ H_4$ is constant if and only if $f_4$ is constant.
	By our convention $\sigma_3\geq 0$, there are no points whose principal curvatures are of the two types: $\lambda_{1}<\lambda_{2}=\lambda_{3}=\lambda_{4}=\lambda_{5}$ and $\lambda_{1}=\lambda_{2}<\lambda_{3}=\lambda_{4}=\lambda_{5}$. Thus we only have the following two cases.
	
	Case (1): One is a triple real root, and the other is a double real root. In this case, we assume $$\lambda_{1}=\lambda_{2}=\lambda_{3}=\zeta<\lambda_{4}=\lambda_{5}=-\frac{3}{2}\zeta.$$
	
	After taking the covariant derivative on both sides of conditions $H=0$ and $S=constant$, 
	we have 	
	$$
	\begin{cases}
		h_{11k} +h_{22k}+h_{33k}+h_{44k}+h_{55k}=0,\\
		\zeta\left(h_{11k} +h_{22k}+h_{33k}\right)-\frac{3}{2}\zeta\left(h_{44k}+h_{55k}\right)=0,
	\end{cases}
	$$
	Thus 
	$$h_{11k} +h_{22k}+h_{33k}=0\ \mbox{and}\ h_{44k}+h_{55k}=0\
	\mbox{for}\ k=1,2,\cdots,5.$$
	
	Since $H=0$ and $S$ is constant, we get $(H)_{mm}=0$ and $(S)_{mm}=0$,  so we have
	\begin{equation}\label{H2}
		\sum_{i=1}^{5}h_{iimm}=0\ \mbox{for}\ m=1,2,\cdots,5.
	\end{equation}
	\begin{equation}\label{S2}
		\sum_{i=1}^{5}h_{ii}h_{iimm}+\sum_{i,j=1}^{5}h_{ijm}^2=0\ \mbox{for}\ m=1,2,\cdots,5.
	\end{equation}
	
		For convenience, we write %$$Y_1=h_{114}^2+h_{115}^2+h_{124}^2+h_{125}^2+h_{134}^2+h_{135}^2.$$
%	Taking $m=1$ in $(\ref{H2})$ and $(\ref{S2})$, we get 
%	\begin{equation}\label{S}
	%	\begin{aligned}
%			\sum_{i=1}^{5}h_{ii}h_{ii11}+\sum_{i,j=1}^{5}h_{ij1}^2=&\zeta\left(h_{1111} +h_{2211}+h_{3311}\right)-\frac{3}{2}\zeta\left(h_{4411}+h_{5511}\right)\\
%			&+h_{111}^2+h_{221}^2+h_{331}^2+h_{441}^2+h_{551}^2\\
%			&+2\left(h_{112}^2+h_{113}^2+h_{123}^2+h_{145}^2+Y_1\right)\\
%			=&-\frac{5}{2}\zeta\left(h_{4411}+h_{5511}\right)+h_{111}^2+h_{221}^2+h_{331}^2+h_{441}^2\\
%			&+h_{551}^2+2\left(h_{112}^2+h_{113}^2+h_{123}^2+h_{145}^2+Y_1\right)\\
%			=&0.
%		\end{aligned}
%	\end{equation}
	$$Y^1_l=h_{14l}^2+h_{15l}^2+h_{24l}^2+h_{25l}^2+h_{34l}^2+h_{35l}^2$$
	for $l=1,2,\cdots,5$. Taking $m=l$ in $(\ref{H2})$ and $(\ref{S2})$, we get 
	\begin{equation}\label{S}
		\begin{aligned}
			\sum_{i=1}^{5}h_{ii}h_{iill}+\sum_{i,j=1}^{5}h_{ijl}^2=&\zeta\left(h_{11ll} +h_{22ll}+h_{33ll}\right)-\frac{3}{2}\zeta\left(h_{44ll}+h_{55ll}\right)\\
			&+h_{11l}^2+h_{22l}^2+h_{33l}^2+h_{44l}^2+h_{55l}^2\\
			&+2\left(h_{12l}^2+h_{13l}^2+h_{23l}^2+h_{45l}^2+Y^1_l\right)\\
			=&-\frac{5}{2}\zeta\left(h_{44ll}+h_{55ll}\right)+h_{11l}^2+h_{22l}^2+h_{33l}^2+h_{44l}^2\\
			&+h_{55l}^2+2\left(h_{12l}^2+h_{13l}^2+h_{23l}^2+h_{45l}^2+Y^1_l\right)\\
			=&0.
		\end{aligned}
	\end{equation}
	
	Next, since $f_4=\sum_{i,j,k,l=1}^{5}h_{ij}h_{jk}h_{kl}h_{li}$ is constant, we obtain $(f_4)_{mm}=0$, i.e.,
	\begin{equation}\label{f42}
		\sum_{i=1}^{5}\lambda_{i}^3h_{iimm}+\sum_{i,j=1}^{5}(2\lambda_{i}^2+\lambda_{i}\lambda_{j})h_{ijm}^2=0\ \mbox{for}\ m=1,2,\cdots,5.
	\end{equation}
	
	Taking $m=l$ in $(\ref{H2})$ and $(\ref{f42})$, we get 
	\begin{equation}\label{f4}
		\begin{aligned}
			\sum_{i=1}^{5}\lambda_{i}^3h_{iill}+&\sum_{i,j=1}^{5}(2\lambda_{i}^2+\lambda_{i}\lambda_{j})h_{ijl}^2=\zeta^3(h_{11ll} +h_{22ll}+h_{33ll})\\
			&-\frac{27}{8}\zeta^3(h_{44ll}+h_{55ll})+\sum_{i,j=1}^{3}3\zeta^2h_{ijl}^2+\sum_{i,j=4}^{5}\frac{27}{4}\zeta^2h_{ijl}^2\\
			&+\sum_{i=1,2,3;j=4,5}(2\zeta^2-\frac{3}{2}\zeta^2)h_{ijl}^2+\sum_{i=4,5;j=1,2,3}(\frac{9}{2}\zeta^2-\frac{3}{2}\zeta^2)h_{ijl}^2\\
			=&-\frac{35}{8}\zeta^3\left(h_{44ll}+h_{55ll}\right)+\zeta^2\Big[3(h_{11l}^2+h_{22l}^2+h_{33l}^2)+6\big(h_{12l}^2+h_{13l}^2\\
			&+h_{23l}^2\big)+\frac{27}{4}\left(h_{44l}^2+h_{55l}^2\right)+\frac{27}{2}h_{45l}^2+\frac{7}{2}Y^1_l\Big]\\
			=&0.
		\end{aligned}
	\end{equation}
	
		Then subtracting $\frac{7}{4}\zeta^2(\ref{S})$ from $(\ref{f4})$, we get 
	$$\frac{5}{4}\left(h_{11l}^2+h_{22l}^2+h_{33l}^2\right)+5\left(h_{44l}^2+h_{55l}^2\right)+\frac{5}{2}\left(h_{12l}^2+h_{13l}^2+h_{23l}^2\right)+10h_{45l}^2=0.$$
	Thus for different $l=1,2,\cdots,5$, we obtain
	$$
	\begin{cases}
		h_{111}=h_{221}=h_{331}=h_{441}=h_{551}=h_{112}=h_{113}=h_{123}=h_{145}=0,\\
		h_{222}=h_{332}=h_{442}=h_{552}=h_{223}=h_{245}=0,\\
		h_{333}=h_{443}=h_{553}=h_{345}=0,\\
		h_{114}=h_{224}=h_{334}=h_{444}=h_{554}=h_{124}=h_{134}=h_{234}=h_{445}=0,\\
		h_{115}=h_{225}=h_{335}=h_{555}=h_{125}=h_{135}=h_{235}=0.\\
	\end{cases} 
	$$

	Case (2): One is a simple real root, and the other is a quadruple real root.  In this case, we assume
	$$\lambda_{1}=\lambda_{2}=\lambda_{3}=\lambda_{4}=\mu<\lambda_{5}=-4\mu.$$
	
	Similar to Case (1), taking the covariant derivatives of $H=0$ and $S=constant$ yields
	
	$$
	\begin{cases}
		h_{11k} +h_{22k}+h_{33k}+h_{44k}+h_{55k}=0,\\
		\mu\left(h_{11k} +h_{22k}+h_{33k}+h_{44k}\right)-4\mu h_{55k}=0,
	\end{cases}
	%\ \mbox{for}\ k=1,2,\cdots,5.
	$$
	thus
	$$h_{11k} +h_{22k}+h_{33k}+h_{44k}=0\ \mbox{and}\ h_{55k}=0\ \mbox{for}\ k=1,2,\cdots,5.$$
	
	Since $f_3=\sum_{i,j,k=1}^{5}h_{ij}h_{jk}h_{ki}$ is constant, we have $(f_3)_{mm}=0$, and we get 
	\begin{equation}\label{case2f3}
		\sum_{i=1}^{5}\lambda_{i}^2h_{iimm}+2\sum_{i,j=1}^{5}\lambda_{i}h_{ijm}^2=0\ \mbox{for}\ m=1,2,\cdots,5.
	\end{equation}
	
		For convenience, we write $$Y^2_l=h_{12l}^2+h_{13l}^2+h_{14l}^2+h_{23l}^2+h_{24l}^2+h_{34l}^2$$
	for $l=1,2,\cdots,5$.
For $m=l$ in $(\ref{H2})$ and $(\ref{case2f3})$, we get 
	\begin{equation}\label{f3}
		\begin{aligned}
			\sum_{i=1}^{5}\lambda_{i}^2h_{iill}+&\sum_{i,j=1}^{5}2\lambda_{i}h_{ijl}^2=\mu^2\left(h_{11ll} +h_{22ll}+h_{33ll}+h_{44ll}\right)\\
			&+16\mu^2h_{55ll}+2\mu\sum_{i}\left(h_{1il}^2+h_{2il}^2+h_{3il}^2+h_{4il}^2-4h_{5il}^2\right)\\
			=&15\mu^2h_{55ll}+2\mu\Big[\left(h_{11l}^2+h_{22l}^2+h_{33l}^2+h_{44l}^2\right)\\
			&+2Y^2_l-3\left(h_{15l}^2+h_{25l}^2+h_{35l}^2+h_{45l}^2\right)\Big]\\
			=&0.
		\end{aligned}
	\end{equation}
	
Then, taking $m=l$ in $(\ref{H2})$ and $(\ref{f42})$, we get 
	\begin{equation}\label{case2f4}
		\begin{aligned}
			\sum_{i=1}^{5}\lambda_{i}^3h_{iill}+&\sum_{i,j=1}^{5}(2\lambda_{i}^2+\lambda_{i}\lambda_{j})h_{ijl}^2=\mu^3\left(h_{11ll} +h_{22ll}+h_{33ll}+h_{44ll}\right)\\
			&-64\mu^3h_{55ll}+\sum_{i,j=1}^{4}3\mu^2h_{ijl}^2+48\mu^2h_{55l}^2\\
			&+\sum_{i=1}^4(-2\mu^2)h_{5il}^2+\sum_{j=1}^4(28\mu^2)h_{5jl}^2\\
			=&-65\mu^3h_{55ll}+\mu^2\Big[3\left(h_{11l}^2+h_{22l}^2+h_{33l}^2+h_{44l}^2\right)\\
			&+6Y^2_l+26\left(h_{15l}^2+h_{25l}^2+h_{35l}^2+h_{45l}^2\right)\Big]\\
			=&0.
		\end{aligned}
	\end{equation}

Next, adding $\frac{13}{3}\mu(\ref{f3})$ and $(\ref{case2f4})$, we obtain

	$$\left(h_{11l}^2+h_{22l}^2+h_{33l}^2+h_{44l}^2\right)
	+2Y^2_l=0.$$
	Thus for different $l=1,2,\cdots,5$, we obtain
	$$
	\begin{cases}
		h_{111}=h_{221}=h_{331}=h_{441}=h_{112}=h_{113}=h_{114}=h_{123}=h_{124}=h_{134}=0,\\
		h_{222}=h_{332}=h_{442}=h_{223}=h_{224}=h_{234}=0,\\
		h_{333}=h_{443}=h_{334}=0,\\
		h_{444}=0,\\
		h_{115}=h_{225}=h_{335}=h_{445}=h_{125}=h_{135}=h_{145}=h_{235}=h_{245}=h_{345}=0.
	\end{cases}
	$$
	%$$h_{114}=h_{123}=h_{124}=h_{134}=0$$
	%For $m=2$, we can get
%	$$h_{222}=h_{332}=h_{442}=h_{223}=h_{224}=h_{234}=0$$
	%	For $m=3$, we can get
%	$$h_{333}=h_{331}=h_{443}=h_{334}=0$$
	%	For $m=4,5$, we can get
%	$$h_{444}=0=h_{55k}=h_{225}=h_{335}=h_{445}=h_{235}=h_{245}=h_{345}=0$$
%	for $k=1,2,3,4,5$.
%	$h_{115}=0$ follows from 

	Finally, combining this with $h_{55k}=0$ for arbitrary $k=1,2,\cdots,5$, we can obtain that all $h_{ijk}=0$ for $i,j,k=1,2,\cdots,5.$
	
	In summary, for the above two cases, we have proved that $h_{ijk}=0$ for all $i,j,k=1,2,\cdots,5$. Furthermore, from $(\ref{DeltaS})$, we have $S(S-5)=0$. Noting that $S>0$, it follows that $S\equiv5$ and $M^5$ is either the Clifford torus $\mathbb{S}^2(\sqrt\frac{2}{5})\times\mathbb{S}^3(\sqrt\frac{3}{5})$ or $\mathbb{S}^1(\sqrt\frac{1}{5})\times\mathbb{S}^4(\sqrt\frac{4}{5})$.
\end{proof}

\hspace*{2em}
\\
\textbf{Acknowledgments.} The author would like to thank Professor Tang Zizhou and Professor Ge Jianquan for helpful discussion.

\begin{appendix}
	\section{Explicit Expression of $\mathcal{P}_{A(3)}$, $\mathcal{P}_{v_1}$ and $\mathcal{P}_{A(5)}$}\label{A}
		$$\begin{aligned}
		\mathcal{P}_{v_1}=&
		\lambda_{2}^2 \lambda_{3}^2 - 2 \lambda_{1}^2 \lambda_{3}^2 - 2 \lambda_{1}^2 \lambda_{2}^2 + \lambda_{1} \lambda_{2}^3 + \lambda_{1}^3 \lambda_{2} + \lambda_{1} \lambda_{3}^3 + \lambda_{1}^3 \lambda_{3} + \lambda_{2} \lambda_{3}^3 - \lambda_{1}^3 \lambda_{4} \\
		&+ \lambda_{2}^3 \lambda_{3} - \lambda_{1}^3 \lambda_{5} - \lambda_{1} \lambda_{2} \lambda_{3}^2 - \lambda_{1} \lambda_{2}^2 \lambda_{3} - \lambda_{1}^2 \lambda_{2} \lambda_{3} - \lambda_{1} \lambda_{2}^2 \lambda_{4} + 2 \lambda_{1}^2 \lambda_{2} \lambda_{4} \\
		&- \lambda_{1} \lambda_{2}^2 \lambda_{5} - \lambda_{1} \lambda_{3}^2 \lambda_{4} + 2 \lambda_{1}^2 \lambda_{2} \lambda_{5} + 2 \lambda_{1}^2 \lambda_{3} \lambda_{4} - \lambda_{1} \lambda_{3}^2 \lambda_{5} - \lambda_{2} \lambda_{3}^2 \lambda_{4} \\
		&+ 2 \lambda_{1}^2 \lambda_{3} \lambda_{5} - \lambda_{2}^2 \lambda_{3} \lambda_{4} - \lambda_{2} \lambda_{3}^2 \lambda_{5} - 3 \lambda_{1}^2 \lambda_{4} \lambda_{5} - \lambda_{2}^2 \lambda_{3} \lambda_{5} + \lambda_{1} \lambda_{2} \lambda_{3} \lambda_{4} \\
		&+ \lambda_{1} \lambda_{2} \lambda_{3} \lambda_{5} + \lambda_{1} \lambda_{2} \lambda_{4} \lambda_{5} + \lambda_{1} \lambda_{3} \lambda_{4} \lambda_{5} + \lambda_{2} \lambda_{3} \lambda_{4} \lambda_{5}.
	\end{aligned}$$
	
	$$\mathcal{P}_{A(3)}=2\lambda_{1}^{2}+5\lambda_{1}\lambda_{2}+6\lambda_{1}\lambda_{5}+2\lambda_{2}^{2}+6\lambda_{2}\lambda_{5}+4\lambda_{5}^{2}.$$
		$$\mathcal{P}_{A(5)}=
	2\alpha_{1}^{2}+6\alpha_{1}\alpha_{3}+5\alpha_{1}\alpha_{4}+4\alpha_{3}^{2}+6\alpha_{3}\alpha_{4}+2\alpha_{4}^{2}.$$

	\section{Expanded Form of $L(5)$}\label{A1}
	For the convenience of verification, we provide the fully expanded form of $L(5)$. The definition of $L(5)$ is given in (\ref{letL}) and (\ref{L}). Explicitly, 
$$\begin{aligned}
L(5) = &-2 \lambda_5 \left( \frac{1}{c_1} + \frac{1}{c_2} + \frac{1}{c_3} + \frac{1}{c_4} + \frac{1}{c_5} + \frac{1}{c_6} \right) +  (b_1 + b_2 + b_3 + b_4) \\
&+  \frac{1}{(\lambda_1 - \lambda_5)(\lambda_2 - \lambda_5)(\lambda_3 - \lambda_5)(\lambda_4 - \lambda_5)} \left( \frac{1}{(\lambda_5 - \lambda_1)^2 (\lambda_2 - \lambda_1)(\lambda_3 - \lambda_1)(\lambda_4 - \lambda_1)} \right.\\
&\left.+ \frac{1}{(\lambda_5 - \lambda_2)^2 (\lambda_4 - \lambda_2)(\lambda_3 - \lambda_2)(\lambda_1 - \lambda_2)} + \frac{1}{(\lambda_5 - \lambda_3)^2 (\lambda_4 - \lambda_3)(\lambda_2 - \lambda_3)(\lambda_1 - \lambda_3)} \right. \\
&\left.+ \frac{1}{(\lambda_5 - \lambda_4)^2 (\lambda_3 - \lambda_4)(\lambda_2 - \lambda_4)(\lambda_1 - \lambda_4)} \right),
\end{aligned}$$
where
$$\begin{aligned}
	c_1 =& (\lambda_5 - \lambda_1)^2 (\lambda_5 - \lambda_2)^2 (\lambda_2 - \lambda_1)^2 (\lambda_3 - \lambda_1) (\lambda_4 - \lambda_1) (\lambda_3 - \lambda_2) (\lambda_4 - \lambda_2),\\
	c_2 =& (\lambda_5 - \lambda_1)^2 (\lambda_5 - \lambda_3)^2 (\lambda_3 - \lambda_1)^2 (\lambda_2 - \lambda_1) (\lambda_4 - \lambda_1) (\lambda_2 - \lambda_3) (\lambda_4 - \lambda_3),\\
	c_3 = &(\lambda_5 - \lambda_1)^2 (\lambda_5 - \lambda_4)^2 (\lambda_4 - \lambda_1)^2 (\lambda_2 - \lambda_1) (\lambda_3 - \lambda_1) (\lambda_2 - \lambda_4) (\lambda_3 - \lambda_4),\\
	c_4 = &(\lambda_5 - \lambda_2)^2 (\lambda_5 - \lambda_3)^2 (\lambda_3 - \lambda_2)^2 (\lambda_1 - \lambda_2) (\lambda_4 - \lambda_2) (\lambda_1 - \lambda_3) (\lambda_4 - \lambda_3),\\
	c_5 =& (\lambda_5 - \lambda_2)^2 (\lambda_5 - \lambda_4)^2 (\lambda_4 - \lambda_2)^2 (\lambda_1 - \lambda_2) (\lambda_3 - \lambda_2) (\lambda_1 - \lambda_4) (\lambda_3 - \lambda_4),\\
	c_6 =& (\lambda_5 - \lambda_3)^2 (\lambda_5 - \lambda_4)^2 (\lambda_4 - \lambda_3)^2 (\lambda_1 - \lambda_3) (\lambda_2 - \lambda_3) (\lambda_1 - \lambda_4) (\lambda_2 - \lambda_4)
\end{aligned}
$$
and 
$$\begin{aligned}
	b_1 =&\frac{1}{(\lambda_2 - \lambda_1)^2 (\lambda_3 - \lambda_1)^2 (\lambda_4 - \lambda_1)^2 (\lambda_5 - \lambda_1)^3},\\
	b_2 =& \frac{1}{(\lambda_2 - \lambda_1)^2 (\lambda_3 - \lambda_2)^2 (\lambda_4 - \lambda_2)^2 (\lambda_5 - \lambda_2)^3},\\
	b_3 =& \frac{1}{(\lambda_2 - \lambda_3)^2 (\lambda_3 - \lambda_1)^2 (\lambda_4 - \lambda_3)^2 (\lambda_5 - \lambda_3)^3},\\
	b_4 =& \frac{1}{(\lambda_2 - \lambda_4)^2 (\lambda_4 - \lambda_1)^2 (\lambda_4 - \lambda_3)^2 (\lambda_5 - \lambda_4)^3}.
\end{aligned}$$
	\section{Complete Expansions of the Derivatives of $\mathcal{P}_5$.}\label{B}
Recall that in the proof of Case $(4)$ of Corollary \ref{cor}, we need the following computation results.
$$
\begin{aligned}
\frac{\partial^6 \mathcal{P}_5}{\partial d^6}=&	18000a^2c-2880ab^2+93600abc+38880ac^2+100800acd\\
&-14400b^3+78480b^2c-20160b^2d+96480bc^2+181440bcd\\
&+6480c^3+80640c^2d+80640cd^2.
\end{aligned}
$$

$$
\begin{aligned}
\left.\frac{\partial^5 \mathcal{P}_5}{\partial d^5}\right|_{d=c}=&	\mathbf{-2400a^3c-600a^2b^2}+15360a^2bc+14040a^2c^2\mathbf{-1680ab^3}\\
&+38160ab^2c+128400abc^2+84720ac^3\mathbf{-3960b^4}+6000b^3c\\
&+114240b^2c^2+197520bc^3\mathbf{+70200c^4}.
\end{aligned}
$$

$$
\begin{aligned}
	\left.\frac{\partial^4 \mathcal{P}_5}{\partial d^4}\right|_{d=c}=&	\mathbf{-1392a^4c-48a^3b^2-2784a^3bc-7200a^3c^2-288a^2b^3}+5976a^2b^2c+10512a^2bc^2\\&\mathbf{-1752a^2c^3-336ab^4}+10752ab^3c+55200ab^2c^2+76128abc^3+27024ac^4\\&\mathbf{-480b^5}+600b^4c+27120b^3c^2+86736b^2c^3+85008bc^4\mathbf{+19224c^5}.
\end{aligned}
$$

$$
\begin{aligned}
	\left.\frac{\partial^3\mathcal{P}_5}{\partial d^3}\right|_{d=c}=&	\mathbf{-120a^5c-6a^4b^2-1236a^4bc-2058a^4c^2-12a^3b^3-1452a^3b^2c-7128a^3bc^2}\\&\mathbf{-7308a^3c^3-30a^2b^4}+1776a^2b^3c+3822a^2b^2c^2\mathbf{-3360a^2bc^3-7008a^2c^4}\\&\mathbf{-24ab^5}+2748ab^4c+16632ab^3c^2+32904ab^2c^3+23136abc^4+3204ac^5\\&-24b^6+492b^5c+6012b^4c^2+23484b^3c^3+\mathbf{38142b^2c^4+24012bc^5+3402c^6}.
\end{aligned}
$$

$$
\begin{aligned}
	\left.\frac{\partial^2 \mathcal{P}_5}{\partial d^2}\right|_{d=c}=&	2c\left[25a^6\mathbf{-22a^5b}+52a^5c\mathbf{-289a^4b^2-792a^4bc-444a^4c^2-242a^3b^3}\right.\\&\mathbf{-1718a^3b^2c-3238a^3bc^2-1752a^3c^3}+247a^2b^4+558a^2b^3c\\&\mathbf{-1065a^2b^2c^2-3392a^2bc^3-2023a^2c^4}+316ab^5+2106ab^4c\\&+5050ab^3c^2+4860ab^2c^3+1136abc^4\mathbf{-468ac^5}+77b^6+722b^5c\\
	&\left.\mathbf{+2777b^4c^2+5444b^3c^3+5342b^2c^4+2196bc^5+162c^6}\right].
\end{aligned}
$$

\end{appendix}

\end{document}